\documentclass[a4paper]{article}

\usepackage{amsmath}
\usepackage{amsthm}
\usepackage{amssymb}

\usepackage{graphicx}
\usepackage{subfigure}

\usepackage{multirow}
\usepackage{rotating}

\usepackage{algpseudocode}
\usepackage{algorithm}

\usepackage{color}
\usepackage[normalem]{ulem}
\usepackage[utf8]{inputenc}
\usepackage{tikz}
\usetikzlibrary{automata}
\usetikzlibrary{arrows,shapes}

\newtheorem{theorem}{Theorem}
\newtheorem{lemma}{Lemma}
\newtheorem{corollary}{Corollary}

\allowdisplaybreaks

\newcommand{\X}{\mathcal{X}}

\newcommand{\cU}{\mathcal{U}}

\newcommand{\cI}{\mathcal{I}}
\newcommand{\cJ}{\mathcal{J}}

\newcommand{\oc}{\overline{c}}
\newcommand{\hc}{\hat{c}}

\newcommand{\BIGOP}[1]{\mathop{\mathchoice%
{\raise-0.22em\hbox{\huge $#1$}}%
{\raise-0.05em\hbox{\Large $#1$}}{\hbox{\large $#1$}}{#1}}}
\newcommand{\bigtimes}{\BIGOP{\times}}

\usepackage{authblk}

\begin{document}

\title{Min-Max Regret Problems with Ellipsoidal Uncertainty Sets\thanks{Effort sponsored by the Air Force Office of Scientific Research, Air Force Material Command, USAF, under
grant number FA8655-13-1-3066. The U.S Government is authorized to reproduce and distribute reprints
for Governmental purpose notwithstanding any copyright notation thereon.}}

\author[1]{Andr\'{e} Chassein\thanks{Email: chassein@mathematik.uni-kl.de}}
\author[2]{Marc Goerigk\thanks{Corresponding author. Email: m.goerigk@lancaster.ac.uk}}

\date{}
\affil[1]{Fachbereich Mathematik, Technische Universit\"at Kaiserslautern, Germany}
\affil[2]{Department of Management Science, Lancaster University, United Kingdom}

\maketitle

\begin{abstract}
We consider robust counterparts of uncertain combinatorial optimization problems, where the difference to the best possible solution over all scenarios is to be minimized. Such minmax regret problems are typically harder to solve than their nominal, non-robust counterparts. While current literature almost exclusively focuses on simple uncertainty sets that are either finite or hyperboxes, we consider problems with more flexible and realistic ellipsoidal uncertainty sets. We present complexity results for the unconstrained combinatorial optimization problem and for the shortest path problem. To solve such problems, two types of cuts are introduced, and compared in a computational experiment.
\end{abstract}

{\bf Keywords:} robust optimization; minmax regret; ellipsoidal uncertainty; complexity; scenario relaxation

\section{Introduction}

We consider general combinatorial optimization problems of the form
\[ \min \{ c^Tx : x\in \X\subseteq\{0,1\}^n \} \]
where the objective vector $c$ is unknown, and coming from a set $\cU$ of possible realizations.
To find a solution $x$ that still performs well under all possible outcomes of $c$, several robust optimization approaches have been developed (for an overview, we refer to \cite{rosurvey,bertsimas-survey,perf}). 

In this paper, we focus on the minmax regret approach, which is amongst the best-established methods in robust optimization \cite{Inuiguchi1995526,KouYu97,Aissi2009}. The basic idea is to find a solution that minimizes the largest difference to the optimal objective value in each scenario. More formally, we use a robust objective function of the form
\[ Reg(x,\cU) = \max \{ c^Tx - opt(c) : c\in\cU \} \]
with $opt(c)$ being the optimal objective value of the original problem with objective function $c$, and aim at solving the minmax regret problem
\[ \min \{ Reg(x,\cU) : x\in\X\} \]
This problem has been extensively analyzed for finite and hyperbox uncertainty sets. 
Most minmax regret problems of this kind are NP-hard, see., e.g.,  \cite{Averbakh2005227,aissi2005complexity,averbakh2001complexity} and the overview in \cite{Aissi2009}. Therefore, both approximation algorithms and heuristic algorithms without performance guarantees have been suggested.

\cite{Kasperski2006} showed that solving the midpoint scenario of an interval uncertainty set gives a 2-approximation for minmax regret combinatorial optimization problems. This was further extended in \cite{Conde2012452} to symmetry points of general uncertainty sets. In \cite{ChasseinGoerigk2014}, an a-posteriori bound for the midpoint solution was presented, which can be used in a branch-and-bound algorithm.

\cite{Montemanni2005771} developed a branch-and-bound algorithm for robust spanning trees. For the same problem, also a scenario relaxation procedure was presented in \cite{Perez14}. The basic idea of scenario relaxation is to begin with a finite subset of scenarios, instead of the whole interval set. Then, worst-case scenarios are iteratively added to the scenario set, until the objective value of this relaxation coincides with the actual objective value of the regret problem with intervals.

Quite surprisingly, little attention has been paid to uncertainty sets that are not finite or hyperboxes. It seems that this is at odds with the development of other approaches to robust optimization, where the use of more sophisticated sets has been of primary importance. We mention ellipsoidal uncertainty sets (see \cite{RObook}) and $\Gamma$-uncertainty sets (see \cite{BertSim04}) as the most prominent examples.

There are several reasons to use ellipsoidal uncertainty sets in robust optimization. First, they give good tractability results for other robust optimization approaches. So far, this question is open for minmax regret. Second, they are flexible, as the generalized $\cap$-ellipsoidal uncertainty introduced in \cite{ben1998robust} even incorporates finite (via their convex hull) and interval sets. 
Third, they are well-motivated from a stochastic setting, where they naturally occur when a normal distribution is cut off at a certain level of probability.

In this paper we consider ellipsoidal uncertainty sets in minmax regret problems. To the best of our knowledge, there is only one previous paper that also considers this type of problem \cite{Takeda2010}. There, the authors consider uncertain convex quadratic problems and present a relaxation heuristic with probability guarantees. In this paper, we focus on combinatorial problems, complexity results and exact solution algorithms.


In Section~\ref{sec:complex} we present complexity results for the unconstrained combinatorial problem, and for the shortest path problem. While the unconstrained problem with finite sets is NP-hard to solve, and the regret objective value of a candidate solution can be computed in polynomial time, we find the surprising result that the reverse holds true for axis-parallel ellipsoids: While it is NP-hard to compute the regret objective of one candidate solution, the optimal solution of the problem can be found in polynomial time.

In Section~\ref{sec:exact}, we discuss two different ways to reformulate the minmax regret problem via a scenario relaxation procedure, resulting in exact, general solution approaches. These algorithms are compared in computational experiments in Section~\ref{sec:experiments}. Final conclusions are drawn and further research directions are posted in Section~\ref{sec:conclusion}. 

\section{Complexity Results}
\label{sec:complex}

\subsection{Problem Definition}

We choose two combinatorial optimization problems to investigate the computational complexity of the minmax regret problem for different uncertainty sets. We consider the cases of
\begin{itemize}
\item interval or hyperbox uncertainty $\mathcal{U}= \bigtimes_{i=1}^n [\hc_i-d_i,\hc_i+d_i]$,
\item axis-parallel ellipsoids $\mathcal{U}= \{c: (c-\hc)^T D (c-\hc)^T \leq 1 \}$, where $D \succ 0$ is a positive definite diagonal matrix,
\item general ellipsoids $\mathcal{U}= \{\hc + C\xi: \|\xi\|_2 \leq 1\}$, and 
\item finite uncertainty sets $\mathcal{U}=\{c^1,\dots,c^k\}$, where $k$ is polynomially bounded.
\end{itemize}

Note that each axis-parallel ellipsoid can be expressed as a general ellipsoid. For each uncertainty set we study the complexity of solving the minmax regret problem, i.e., finding the optimal solution of  
\[\min_{x\in\X} Reg(x,\cU) = \min_{x \in \mathcal{X}} \max_{c \in \mathcal{U}}  \left(c^Tx - \min_{y \in \X} c^Ty \right) \tag{Solve} \]

\noindent
and evaluating the regret of a given solution, i.e., computing the value of
\[Reg(x,\cU) = \max_{c \in \mathcal{U}} \left( c^Tx - \min_{y \in \X} c^Ty\right). \tag{Eval}\]

The unconstrained combinatorial problem is the simplest non-trivial combinatorial problem. The feasible set $\mathcal{X}=\{0,1\}^n$ is the set of all $0,1$-vectors. We denote this problem as $(UP)$.
The shortest path problem is one of the most studied combinatorial problems. Each vector $x$ in the feasible set $\mathcal{X}$ is an incidence vector of an $s-t$ paths in a graph~$G$. This problem is denoted as $(SP)$.

Some of the presented reductions use the \textsf{NP}-complete \emph{partition problem}: Given a list of natural numbers $a_1,\dots,a_n$. The problem is to decide if a subset $I\subset \{1,\dots,n\}$ of the index set exists such that $\sum_{i \in I} a_i = \sum_{i \notin I} a_i$.

The following lemmas are used in some of the proofs.

\begin{lemma}{(See \cite{ben1999robust})}\label{lem1a}
For an ellipsoidal uncertainty set $\mathcal{U}= \{C\xi+ \hc: \|\xi\|_2 \leq 1\}$, it holds that
\begin{equation}
\max_{c\in \mathcal{U}} c^T x = \hc^T x + \| C^Tx \|_2.
\end{equation}
\end{lemma}

In case of an axis-parallel ellipsoid, Lemma~\ref{lem1a} becomes:

\begin{lemma}
For an axis-parallel ellipsoidal uncertainty set $\mathcal{U}= \{c: (c-\hc)^T D (c-\hc)^T \leq 1 \}$, it holds that
\begin{equation}
\max_{c\in \mathcal{U}} c^Tx = \hc^Tx + \sqrt{\sum_{i=1}^n D^{-1}_{ii} x_i^2} \label{lemma1}
\end{equation}
\end{lemma}

\subsection{The Unconstrained Combinatorial Problem}

Robust counterparts of the unconstrained combinatorial problem were first considered in \cite{robunc}, where it was shown that the minmax counterpart
\[ \min_{x\in\{0,1\}^n} \max_{c\in\cU} c^T x \]
is \textsf{NP}-hard already for an uncertainty set consisting only of two scenarios. To the best of our knowledge, no complexity results have been provided for the minmax regret counterpart. For the sake of completeness,
we therefore also consider the complexity for interval and finite sets in this section.

We first consider the evaluation problem of $(UP)$. For interval uncertainty and finite uncertainty, evaluating a solution is simple, as the following two theorems demonstrate.

\begin{theorem}
The evaluation problem of $(UP)$ for interval uncertainty sets is in~\textsf{P}.
\label{thm_ucp_eval_int}
\end{theorem}
\begin{proof}
The evaluation problem is given by
\begin{align}
\begin{split}
\max_{c \in \mathcal{U}} \left( c^Tx - \min_{y\in \mathcal{X}} c^Ty \right)
&= \max_{c \in \mathcal{U}} \left( c^Tx - \sum_{i=1}^n \min(0,c_i) \right) \\
&= \sum_{i=1}^n c^*_i(x_i)x_i - \sum_{i=1}^n \min(0,c^*_i(x_i))
\end{split}\label{th1-eq}
\end{align}
with $c^*_i(x_i) = \hc_i + (2x_i - 1) d_i$. For fixed $x$ this expression can be computed in~$O(n)$.
\end{proof}

\begin{theorem}
The evaluation problem of $(UP)$ for finite uncertainty sets is in~\textsf{P}.
\label{thm_ucp_eval_finite}
\end{theorem}
\begin{proof}
The evaluation problem is given by
\begin{align*}
\max_{c \in \mathcal{U}} \left( c^Tx - \min_{y\in \mathcal{X}} c^Ty \right) 
&= \max_{c \in \mathcal{U}} \left( c^Tx - \sum_{i=1}^n \min(0,c_i) \right) \\
&= \max_{j=1,\dots,k} \left( \sum_{i=1}^n c_i^jx_i - \sum_{i=1}^n \min(0,c_i^j) \right)
\end{align*}
For fixed $x$ this expression can be computed in $O(nk)$. 
\end{proof}

We now turn to the more involved case of ellipsoidal uncertainty sets. Here, evaluating a solution is already a hard problem, as the following theorem shows.

\begin{theorem}\label{theo:up-axis}
The evaluation problem of $(UP)$ for axis-parallel ellipsoidal uncertainty sets is \textsf{NP}-complete.
\label{thm_ucp_eval_ell}
\end{theorem}
\begin{proof}
We give a reduction from the partition problem.

The axis-parallel ellipsoidal uncertainty set $\mathcal{U}$ is defined by the midpoint vector $\hc$ and diagonal matrix $D$. We set $\hc_i=2a_i$ and $D_i=\frac{1}{8Aa_i}$ for $i=1,\dots,n$ and $A=\sum_{i=1}^n a_i$. Consider the evaluation problem for $x=0$:
\begin{align*}
\max_{c \in \mathcal{U}} \left( c^Tx - \min_{y\in \mathcal{X}} c^Ty \right)
&= \max_{c \in \mathcal{U}} \left( 0 - \min_{y\in \mathcal{X}} c^Ty \right) \\
&= \max_{c \in \mathcal{U}} \max_{y\in \mathcal{X}} c^T(-y) \\
&= \max_{y\in \mathcal{X}} \max_{c \in \mathcal{U}}  c^T(-y) \\
&\stackrel{\text{Eq.~\eqref{lemma1}}}{=} \max_{y\in \mathcal{X}} \sum_{i=1}^n 2a_i(-y_i) + \sqrt{\sum_{i=1}^n 8Aa_i(-y_i)^2} \\ 
&= -\min_{y\in \mathcal{X}} \sum_{i=1}^n 2a_iy_i - \sqrt{\sum_{i=1}^n 8Aa_iy_i}
\end{align*}

Define for each solution $y \in \mathcal{X}$ the value $\lambda_y := \frac{1}{A}\sum_{i=1}^n a_iy_i$. Note that the objective value of the minimization problem can be expressed using $\lambda_y$

\begin{align*}
\sum_{i=1}^n 2a_iy_i - \sqrt{\sum_{i=1}^n 8Aa_iy_i} = 2A\lambda_y - \sqrt{8A^2\lambda_y}
\end{align*}

Consider the function $f:[0,1] \rightarrow \mathbb{R}, f(\lambda)=2A\lambda - \sqrt{8A^2 \lambda}$. The minimum of this function is attained for $\lambda^* = 0.5$ due to the first order condition, further $f(\lambda^*)=-A$. This observation proves that the regret for $x=0$ is at least $A$ if and only if the partition instance is a yes-instance. 
\end{proof}

As a direct consequence of Theorem~\ref{theo:up-axis}, we also have that the general case is \textsf{NP}-complete.
\begin{corollary}\label{cor_ucp_eval_ell}
The evaluation problem of $(UP)$ for general ellipsoidal uncertainty sets is \textsf{NP}-complete.
\end{corollary}

Having established the complexity of the evaluation problem, we now turn to the solution problem. We first consider the complexity for finite uncertainty sets.

\begin{theorem}
The solution problem for finite uncertainty sets is \textsf{NP}-complete.
\label{thm_ucp_sol_finite}
\end{theorem}
\begin{proof}
Again we use a reduction from partition. The uncertainty set consists of only two scenarios $c^1=(a_1,\dots,a_n)$ and $c^2=(-a_1,\dots,-a_n)$. Denote by $A=\sum_{i=1}^n a_i$. We claim that a solution with regret at most $\frac{A}{2}$ exists if and only if the partition instance is a yes instance. Let $I$ be the solution of the partition instance. We define solution $x^*_i=1 \ \forall i\in I$ and $x^*_i=0 \ \forall i \notin I$. The regret of $x^*$ is given by
\begin{align*}
\max(\sum_{i\in I} a_i,-\sum_{i \in I} a_i +A)= \max(\sum_{i\in I} a_i,\sum_{i\notin I} a_i) = \frac{A}{2}.
\end{align*}
Conversely, let a solution $x$ with regret at most $\frac{A}{2}$ be given. Let $S=\{i:x_i=1 ,i=1,\dots,n\}$. The regret of $x$ is given by
\begin{align*}
\max(\sum_{i\in S} a_i,-\sum_{i \in S} a_i +A)= \max(\sum_{i\in S} a_i,\sum_{i\notin S} a_i) \geq \frac{A}{2}.
\end{align*}
Since the regret of $x$ is at most $\frac{A}{2}$, we know that the regret of $x$ is exactly~$\frac{A}{2}$. Therefore, $\max(\sum_{i\in S} a_i,\sum_{i\notin S} a_i) = \frac{A}{2}$, which proves that $S$ is a solution for the partition instance. 
\end{proof}

Instead of considering the case of interval and axis-parallel ellipsoid uncertainty sets separately, we directly consider the more general case of \emph{axis-symmetric uncertainty sets}. A set $\mathcal{U}$ is axis-symmetric if it exists a midpoint $\hat{c} \in \mathcal{U}$ such that for any $c\in \mathcal{U}$ with $c=\hat{c} + \gamma$ for any index $i$ it holds that $c^i:=c - 2\gamma_i e_i \in \mathcal{U}$ where $e_i$ is the $i^{\text{th}}$ unit vector. Prominent axis-symmetric uncertainty sets are interval, axis-parallel ellipsoids, or $\Gamma-$uncertainty sets.

\begin{theorem}
The midpoint solution
\[ \hat{x} \in \arg \min \{ \hat{c}^T x : x\in\X \} \]
is an optimal solution of $(UP)$ for axis-symmetric uncertainty sets.
\label{thm_ucp_sol_axis_par}
\end{theorem}
\begin{proof}
We define $\hat{x}_i=1$ if and only if $\hc_i\leq 0$. The goal is to show that $\hat{x}$ is optimal for the minmax regret problem. Let $x^*$ be an optimal solution with $x^*_i=1$ and $\hc_i>0$ for some $i$. In the following we show that $x'=x^*-e_i$ is also an optimal solution.
\begin{align*}
\operatorname{Reg}(x') &= \max_{c\in \mathcal{U}} \max_{y\in \mathcal{X}} c^T(x'-y) \\
&= \max_{c\in \mathcal{U}} \left( c^Tx' -\sum_{i=1}^n \min(0,c_i) \right) \\
&= c'^Tx' -\sum_{i=1}^n \min(0,c'_i) \\
&= \hc^Tx' + \gamma'^Tx' -\sum_{i=1}^n \min(0,\hc_i + \gamma'_i)
\end{align*}
where $c'$ is the worst case scenario (for the regret objective function) and $c'=\hc+\gamma'$. We define $\tilde{\gamma}_j = \gamma'_j \ \forall j \neq i$ and $\tilde{\gamma}_i = - \gamma'_i$. We claim that
\begin{align*}
 \hc^Tx' + \gamma'^Tx' -\sum_{i=1}^n \min(0,\hc_i + \gamma'_i) \leq  \hc^Tx^* + \tilde{\gamma}^Tx^* -\sum_{i=1}^n \min(0,\hc_i + \tilde{\gamma}_i) \tag{$*$}
\end{align*}
Using $(*)$ we can show that $x'$ is also an optimal solution, since
\begin{align*}
\operatorname{Reg}(x') &= \hc^Tx' + \gamma'^Tx' -\sum_{i=1}^n \min(0,\hc_i + \gamma'_i) \\
&\leq  \hc^Tx^* + \tilde{\gamma}^Tx^* -\sum_{i=1}^n \min(0,\hc_i + \tilde{\gamma}_i) \leq  \operatorname{Reg}(x^*) 
\end{align*}
Simplifying $(*)$ yields
\begin{align*}
-\min(0,\hc_i+\gamma'_i) &\leq \hc_i  -\gamma'_i - \min(0,\hc_i-\gamma'_i) \\
\Leftrightarrow \max(0,-\hc_i-\gamma'_i) &\leq \max(0,\hc_i-\gamma'_i) 
\end{align*}
which is true since $0 \leq \hc_i$. The other direction is analog: If $\hc_i\leq 0$ and $x^*_i=0$, $x'=x^*+e_i$ is also an optimal solution. Both directions together show that $\hat{x}$ is an optimal solution of the minmax regret problem.
\end{proof}


For ellipsoidal uncertainty sets, we find the surprising result that while it is a difficult task to evaluate the objective value of a solution, finding a solution with the best possible objective value is simple. However, for general ellipsoids, the solution problem is \textsf{NP}-hard, as the following result states. 

%

\begin{theorem}
The solution problem  of $(UP)$ for ellipsoidal uncertainty sets is \textsf{NP}-hard.
\label{thm_ucp_sol_ell}
\end{theorem}
\begin{proof}
The idea of this proof is to build a degenerated ellipsoid which corresponds to the line segment between the two scenarios $c^1$ and $c^2$ used in the proof of Theorem~\ref{thm_ucp_sol_finite}. Denote by $\mathcal{L}$ the line between $c^1$ and $c^2$. Note that $\max_{c\in \mathcal{L}} \max_{y\in \X} c^T(x-y) = \max_{c\in \{c^1,c^2\}} \max_{y\in \X} c^T(x-y)$.
\end{proof}

\noindent
We summarize the complexity results of this section in Table~\ref{tab:up}.

\begin{table}[h]
\makebox[\textwidth][c]{
\begin{tabular}{|c|c|c|c|c|}
\hline
& Interval & Finite  & Axis-Parallel Ellipsoid & General Ellipsoid   \\ \hline
Eval & \textsf{P} (Thm.~\ref{thm_ucp_eval_int}) & \textsf{P} (Thm.~\ref{thm_ucp_eval_finite}) & \textsf{NPC} (Thm.~\ref{thm_ucp_eval_ell}) & \textsf{NPC} (Cor.~\ref{cor_ucp_eval_ell}) \\ 
Solve & \textbf{Easy} (Thm.~\ref{thm_ucp_sol_axis_par}) & \textsf{NPC} (Thm.~\ref{thm_ucp_sol_finite}) & \textbf{Easy} (Thm.~\ref{thm_ucp_sol_axis_par}) & \textsf{NPH} (Thm.~\ref{thm_ucp_sol_ell}) \\ \hline
\end{tabular}}
\caption{Overview of the different complexity results of the minmax regret unconstrained combinatorial problem.}\label{tab:up}
\end{table}

\subsection{Shortest Path Problem}

We assume in this section that $\mathcal{U} \subset \mathbb{R}^+_n$ to avoid shortest path problems with negative arc weights, since these problems are already $\textsf{NP}$-hard in general. The complexity of the minmax regret shortest path problem is well-researched for interval and finite uncertainty sets. For a finite, but constant number of scenarios, the problem is $\textsf{NP}$-hard and allows a pseudo-polynomial solution algorithm \cite{yu1998robust}. For a non-constant number of scenarios and in the case of interval uncertainty, the problem is strongly \textsf{NP}-hard \cite{KouYu97,averbakh2004interval}. To evaluate the regret of a solution, we need to solve $k$ shortest path problems in the case of a finite uncertainty set with $k$ scenarios, and only a single shortest path problem in the case of interval uncertainty.

We begin with the evaluation problem for ellipsoidal uncertainty sets in Theorem~\ref{thm_sp_eval_ell}, before considering the solution problem in Theorem~\ref{thm_sp_sol_ell}.

\begin{theorem}
The evaluation problem for (axis-parallel) ellipsoidal uncertainty sets is \textsf{NP}-complete.
\label{thm_sp_eval_ell}
\end{theorem}
\begin{proof}
We use again a reduction from the partition problem. For a given instance $a_1,\dots,a_n$ we define the graph as shown in Figure~\ref{fig:graph_sp_eval}.

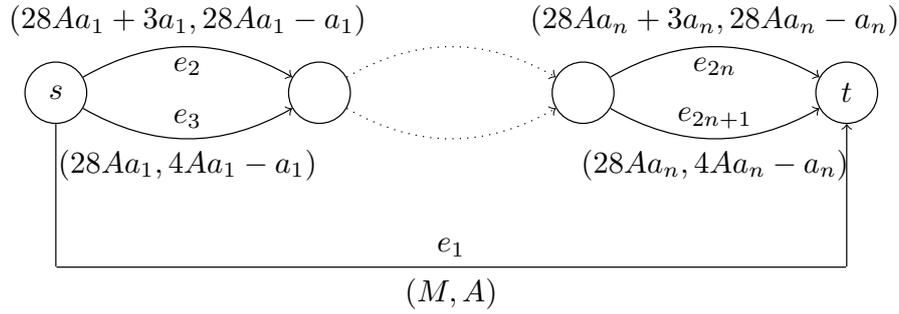
\begin{figure}
\centering
\resizebox{\textwidth}{!}{%
\begin{tikzpicture}[ 
     every state/.style={ 
       minimum size=2em, 
       fill=white, 
       text=black 
     }, 
     node distance=3cm 
   ] 
   
	\tikzstyle{empty} = [circle,fill = white, minimum size = 0pt,inner sep=0pt];

     \node[state] (s)              {$s$}; 
     \node[state] (1) [right of=s] {}; 
     \node[state] (2) [right of=1] {}; 
     \node[state] (t) [right of=2] {$t$}; 

	\node[empty] (e1) at (0,-2) {};
	\node[empty] (e2) at (9,-2) {};

	\node[empty] (e3) at (0,4) {};
	\node[empty] (e4) at (0,5) {};

    \path[->, bend left] (s) edge node[above] {$(28Aa_1+3a_1,28Aa_1-a_1)$} node[below] {$e_2$} (1); 
	\path[->, bend right] (s) edge node[below] {$(28Aa_1,4Aa_1-a_1)$} node[above] {$e_3$} (1); 

    \path[->, bend left, dotted] (1) edge (2); 
    \path[->, bend right, dotted] (1) edge (2);

    \path[->, bend left] (2) edge node[above] {$(28Aa_n+3a_n,28Aa_n-a_n)$} node[below] {$e_{2n}$} (t); 
	\path[->, bend right] (2) edge node[below] {$(28Aa_n,4Aa_n-a_n)$} node[above] {$e_{2n+1}$} (t); 
			
	\draw[-] (s) edge (e1);
	\path[-] (e1) edge node[below] {$(M,A)$} node[above] {$e_1$} (e2); 
	\draw[->] (e2) edge (t);

\end{tikzpicture}
}
\caption{The graph used in the proof of Theorem~\ref{thm_sp_eval_ell}. The labels below and above each edge indicate the number of the edge and the the values $(\hc_e,d_e)$ which describe the uncertainty set.}
\label{fig:graph_sp_eval}
\end{figure}

The pairs $(\hc_e,d_e)$ on each edge define the size of the uncertainty set $\mathcal{U} = \{c: (c-\hc)^TD(c- \hc)\leq 1\}$, where $D$ is implicitly given by  $D_e^{-1}:=d_e$. $M$ is a sufficiently large constant depending on $A$. The set of all edges is denoted by $E'$. The set of all edges except of the first edge is denoted by $E=E' - \{e_1\}$. Note that $\hc_e\geq d_e \ \forall e \in E'$ and $d_e\geq 1 \ \forall e \in E'$. Hence, $\mathcal{U} \subset \mathbb{R}^+_n$. Consider the problem of computing $\operatorname{Reg}(x)$ for $x=(1,0,\dots,0)$, i.e., the path consisting only of the first edge $e_1$. Using Lemma~\ref{lemma1} we can conclude that 
\begin{align*}
\operatorname{Reg}(x) &= \max_{y\in \mathcal{X}} \max_{c\in \mathcal{U}} c^T(x-y) \\
&= \max_{y\in \mathcal{X}}  \left( \hc^T(x-y) + \sqrt{\sum_{e\in E'} d_e (x_e-y_e)^2 } \right) \\
&= \hc^Tx - \min_{y\in \mathcal{X}} \left( \hc^Ty - \sqrt{\sum_{e\in E'} d_e (x_e-y_e)^2 } \right)
\end{align*}
Since $M$ is a large constant we can exclude the solution $y=(1,0,\dots,0)$ without changing the optimal value of the minimization problem. Further, we have that $y_{2k}+y_{2k+1}=1 \ \forall k=1,\dots,n$ due to the structure of the graph. Hence, the problem simplifies to
\begin{align*}
\operatorname{Reg}(x) =& M - \min_{y\in \mathcal{X}} \left( \sum_{e\in E} \hc_e y_e - \sqrt{\sum_{e\in E} d_e y_e + A} \right) \\
=& M - \min_{y\in \mathcal{X}} \Bigg( \sum_{k=1}^n y_{2k}(28Aa_k+3a_k)+ (1-y_{2k})(28Aa_k) \\ 
&- \sqrt{\sum_{k=1}^n y_{2k}(28Aa_k-a_k)+ (1-y_{2k})(4Aa_k-a_k) + A}\ \Bigg) \\
=& M - \min_{y\in \mathcal{X}} \left( 28A^2 + 3\sum_{k=1}^n y_{2k}a_k - \sqrt{4A^2+24A\sum_{k=1}^n y_{2k}a_k} \right)
\end{align*}
Hence, the objective value of each solution $y$ can be expressed by the value $\lambda_y = \frac{1}{A}\sum_{k=1}^n y_{2k} a_k$.
\begin{align*}
\operatorname{Reg}(x) &= M - \min_{y\in \mathcal{X}} \left( 28A^2 + 3A \lambda_y - \sqrt{4A^2+24A^2 \lambda_y} \right)
\end{align*}
Consider the function $f:[0,1] \rightarrow \mathbb{R},$ $f(\lambda) = 28A^2 + 3A \lambda - \sqrt{4A^2+24A^2 \lambda}$. The minimum of this function is attained for $\lambda^* = 0.5$ due to the first order condition, further $f(\lambda^*)=28A^2-2.5A$. Hence, $Reg(x)\geq M - 28A^2+2.5A$ if and only if the partition instance is a yes-instance.
\end{proof}

\begin{theorem}
The solution problem for (axis-parallel) ellipsoidal uncertainty sets is \textsf{NP}-hard.
\label{thm_sp_sol_ell}
\end{theorem}
\begin{proof}
We use a reduction from exact $3$-SAT which is known to be $\textsf{NP}$-complete. We begin the construction by defining the uncertainty set $\mathcal{U}=\{c: (c-\hc)^TD(c-\hc)\leq 1\}$ with diagonal matrix $D$.
We set the average cost of each edge $e$ and the corresponding diagonal entry of $D$ are $1$, i.e., $\hc_e=D_{ee}=1 \ \forall e$. Note that $\mathcal{U} \subset \mathbb{R}^n_+$. Second, all $s-t$ paths consist of $L$ edges. With these restrictions the minmax regret problem can be simplified as follows
\begin{align*}
\min_{x \in \mathcal{X}} \max_{c\in \mathcal{U}} \left( c^Tx - \min_{y \in \mathcal{X}} c^Ty\right) &=\min_{x \in \mathcal{X}} \max_{y \in \mathcal{X}} \max_{c\in \mathcal{U}} c^T(x-y) \\ 
&=\min_{x \in \mathcal{X}} \max_{y \in \mathcal{X}} \left( \hc^T(x-y) + ||x-y||_2 \right)\\
&=\min_{x \in \mathcal{X}} \max_{y \in \mathcal{X}} \left( L - L + ||x-y||_2\right) \\
&=\min_{x \in \mathcal{X}} \max_{y \in \mathcal{X}} \sqrt{x^Tx -2x^Ty+ y^Ty} \\
&=\min_{x \in \mathcal{X}} \max_{y \in \mathcal{X}} \sqrt{2L-2x^Ty}
\end{align*}
For each path represented by $x$ denote by $S(x)=\min_{y \in \mathcal{X}} x^Ty$ the minimum number of edges this path shares with all other $s-t$ paths. Then, $\operatorname{Reg}(x)= \sqrt{2L-2S(x)}$. Therefore, minimizing the regret is equivalent to maximizing $S(x)$. For a given SAT instance, we construct a graph such that $\max_{x\in \mathcal{X}} S(x) \geq 1$ if and only if the SAT instance is a yes-instance. This proves the theorem.

Assume that we are given an instance of $3$-SAT with $n$ literals $l_1,\dots,l_n$ and $m$ clauses $C_1,\dots,C_m$. 
To describe the graph we construct, we use a simple example.
Assume the $3$-SAT instance contains only $3$ literals $l_1,l_2,$ and $l_3$ and a single clause $C_1=(l_1 \vee \overline{l}_2 \vee l_3)$. For clarity, we introduce the graph $G$ in three parts $G_1$,$G_2$, and $G_3$. First we state the part of the graph $G_1$ in Figure~\ref{fig:graph_G1}. The next claims justify to restrict our attention to $G_1$ if we search for a path $x$ maximizing $S(x)$. \\

\noindent
\textbf{1. Claim:} For all paths $x$ in $G$ it holds that $S(x)\leq 1$. \\

\noindent
\textbf{2. Claim:} If a path $x$ in $G$ exists with $S(x)=1$, then there exists also a path $x'$ contained in $G_1$ with $S(x')=1$.\\

The claims are proved at the end of the graph construction, when the complete graph is defined. 

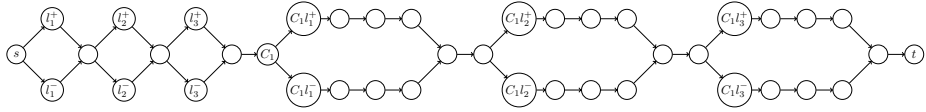
\begin{figure}
\centering
\resizebox{\textwidth}{!}{%
\begin{tikzpicture}
\tikzstyle{every state}=[circle,fill=white!25,minimum size=15pt,inner sep=1pt];
\node[state] (0) at (0,0) {$s$};
\node[state] (1pos) at (1,1) {$l_1^+$};
\node[state] (1neg) at (1,-1) {$l_1^-$};
\node[state] (1con) at (2,0) {};

\node[state] (2pos) at (3,1) {$l_2^+$};
\node[state] (2neg) at (3,-1) {$l_2^-$};
\node[state] (2con) at (4,0) {};

\node[state] (3pos) at (5,1) {$l_3^+$};
\node[state] (3neg) at (5,-1) {$l_3^-$};
\node[state] (3con) at (6,0) {};

\draw [ ->] (0) edge (1pos); 
\draw [ ->] (0) edge (1neg); 
\draw [ ->] (1pos) edge (1con); 
\draw [ ->] (1neg) edge (1con); 

\draw [ ->] (1con) edge (2pos); 
\draw [ ->] (1con) edge (2neg); 
\draw [ ->] (2pos) edge (2con); 
\draw [ ->] (2neg) edge (2con); 

\draw [ ->] (2con) edge (3pos); 
\draw [ ->] (2con) edge (3neg); 
\draw [ ->] (3pos) edge (3con); 
\draw [ ->] (3neg) edge (3con); 

\node[state] (C1start) at (7,0) {$C_1$};

\node[state] (C1pos1) at (8,1) {$C_1l_1^+$};
\node[state] (C1pos2) at (9,1) {};
\node[state] (C1pos3) at (10,1) {};
\node[state] (C1pos4) at (11,1) {};

\node[state] (C1neg1) at (8,-1) {$C_1l_1^-$};
\node[state] (C1neg2) at (9,-1) {};
\node[state] (C1neg3) at (10,-1) {};
\node[state] (C1neg4) at (11,-1) {};

\node[state] (C1end) at (12,0) {};

\node[state] (C2start) at (13,0) {};

\node[state] (C2pos1) at (14,1) {$C_1l_2^+$};
\node[state] (C2pos2) at (15,1) {};
\node[state] (C2pos3) at (16,1) {};
\node[state] (C2pos4) at (17,1) {};

\node[state] (C2neg1) at (14,-1) {$C_1l_2^-$};
\node[state] (C2neg2) at (15,-1) {};
\node[state] (C2neg3) at (16,-1) {};
\node[state] (C2neg4) at (17,-1) {};

\node[state] (C2end) at (18,0) {};

\node[state] (C3start) at (19,0) {};

\node[state] (C3pos1) at (20,1) {$C_1l_3^+$};
\node[state] (C3pos2) at (21,1) {};
\node[state] (C3pos3) at (22,1) {};
\node[state] (C3pos4) at (23,1) {};

\node[state] (C3neg1) at (20,-1) {$C_1l_3^-$};
\node[state] (C3neg2) at (21,-1) {};
\node[state] (C3neg3) at (22,-1) {};
\node[state] (C3neg4) at (23,-1) {};

\node[state] (C3end) at (24,0) {};
\node[state] (t) at (25,0) {$t$};
\draw [ ->] (C3end) edge (t);

\draw [ ->] (3con) edge (C1start); 
\draw [ ->] (C1start) edge (C1pos1); 
\draw [ ->] (C1start) edge (C1neg1);
\draw [ ->] (C1pos1) edge (C1pos2); 
\draw [ ->] (C1pos2) edge (C1pos3); 
\draw [ ->] (C1pos3) edge (C1pos4); 
\draw [ ->] (C1pos4) edge (C1end); 
\draw [ ->] (C1neg1) edge (C1neg2); 
\draw [ ->] (C1neg2) edge (C1neg3); 
\draw [ ->] (C1neg3) edge (C1neg4); 
\draw [ ->] (C1neg4) edge (C1end);

\draw [ ->] (C1end) edge (C2start); 
\draw [ ->] (C2start) edge (C2pos1); 
\draw [ ->] (C2start) edge (C2neg1);
\draw [ ->] (C2pos1) edge (C2pos2); 
\draw [ ->] (C2pos2) edge (C2pos3); 
\draw [ ->] (C2pos3) edge (C2pos4); 
\draw [ ->] (C2pos4) edge (C2end); 

\draw [ ->] (C2neg1) edge (C2neg2); 
\draw [ ->] (C2neg2) edge (C2neg3); 
\draw [ ->] (C2neg3) edge (C2neg4); 
\draw [ ->] (C2neg4) edge (C2end); 

\draw [ ->] (C2end) edge (C3start); 
\draw [ ->] (C3start) edge (C3pos1); 
\draw [ ->] (C3start) edge (C3neg1);
\draw [ ->] (C3pos1) edge (C3pos2); 
\draw [ ->] (C3pos2) edge (C3pos3); 
\draw [ ->] (C3pos3) edge (C3pos4); 
\draw [ ->] (C3pos4) edge (C3end); 

\draw [ ->] (C3neg1) edge (C3neg2); 
\draw [ ->] (C3neg2) edge (C3neg3); 
\draw [ ->] (C3neg3) edge (C3neg4); 
\draw [ ->] (C3neg4) edge (C3end);

\end{tikzpicture}
}
\caption{This part of the graph ($G_1$) is used to represent the literal and clause assignments.}
\label{fig:graph_G1}
\end{figure}

Each path $x$
in $G_1$
represents a literal and clause assignment. The first part of the path from node $s$ to node $C_1$ represents the assignment of the literals. For example: The assignment $l_1=0,l_2=1,l_3=1$ is represented by the path that contains the nodes $l_1^-,l_2^+,$ and $l_3^+$. The second part of the path form node $C_1$ to node $t$ represents how the literals of clause $C_1$ are chosen. If the part contains for example the nodes $C_1l_1^+,C_1l_2^-,$ and $C_1l_3^+$, then we assign the literals $l_1=1,l_2=0,$ and $l_3=1$ in clause $C_1$. Note that all paths in this graph have the same length. Two requirements need to be modeled. First, the assignment of the literals must correspond with the assignments of the literals in each clause and, second, the literal assignment should satisfy all clauses. In the next step we are going to introduce the part $G_2$ and $G_3$ which help to model these requirements. The underlying idea is the following: If one of these two requirements in not fulfilled by the path $x$, there exists another $s-t$ path $y$ (containing edges of $G_2$ or $G_3$) which has no edge in common with $x$, i.e., $S(x)=0$. 

Next we introduce the part $G_2$ which makes sure that the assignment of the literals must be consistent with the assignments of the literals in each clause.

\begin{figure}[h]
\centering
\resizebox{\textwidth}{!}{%
\begin{tikzpicture}
\tikzstyle{every state}=[circle,fill=white!25,minimum size=15pt,inner sep=1pt];
\node[state] (0) at (0,0) {$s$};
\node[state] (1pos) at (1,1) {$l_1^+$};
\node[state] (1neg) at (1,-1) {$l_1^-$};
\node[state] (1con) at (2,0) {};

\node[state] (2pos) at (3,1) {$l_2^+$};
\node[state] (2neg) at (3,-1) {$l_2^-$};
\node[state] (2con) at (4,0) {};

\node[state] (3pos) at (5,1) {$l_3^+$};
\node[state] (3neg) at (5,-1) {$l_3^-$};
\node[state] (3con) at (6,0) {};

\draw [ ->] (0) edge (1pos); 
\draw [ ->] (0) edge (1neg); 
\draw [ ->] (1pos) edge (1con); 
\draw [ ->] (1neg) edge (1con); 

\draw [ ->] (1con) edge (2pos); 
\draw [ ->] (1con) edge (2neg); 
\draw [ ->] (2pos) edge (2con); 
\draw [ ->] (2neg) edge (2con); 

\draw [ ->] (2con) edge (3pos); 
\draw [ ->] (2con) edge (3neg); 
\draw [ ->] (3pos) edge (3con); 
\draw [ ->] (3neg) edge (3con); 

\node[state] (C1start) at (7,0) {$C_1$};

\node[state] (C1pos1) at (8,1) {$C_1l_1^+$};
\node[state] (C1pos2) at (9,1) {};
\node[state] (C1pos3) at (10,1) {};
\node[state] (C1pos4) at (11,1) {};

\node[state] (C1neg1) at (8,-1) {$C_1l_1^-$};
\node[state] (C1neg2) at (9,-1) {};
\node[state] (C1neg3) at (10,-1) {};
\node[state] (C1neg4) at (11,-1) {};

\node[state] (C1end) at (12,0) {};

\node[state] (C2start) at (13,0) {};

\node[state] (C2pos1) at (14,1) {$C_1l_2^+$};
\node[state] (C2pos2) at (15,1) {};
\node[state] (C2pos3) at (16,1) {};
\node[state] (C2pos4) at (17,1) {};

\node[state] (C2neg1) at (14,-1) {$C_1l_2^-$};
\node[state] (C2neg2) at (15,-1) {};
\node[state] (C2neg3) at (16,-1) {};
\node[state] (C2neg4) at (17,-1) {};

\node[state] (C2end) at (18,0) {};

\node[state] (C3start) at (19,0) {};

\node[state] (C3pos1) at (20,1) {$C_1l_3^+$};
\node[state] (C3pos2) at (21,1) {};
\node[state] (C3pos3) at (22,1) {};
\node[state] (C3pos4) at (23,1) {};

\node[state] (C3neg1) at (20,-1) {$C_1l_3^-$};
\node[state] (C3neg2) at (21,-1) {};
\node[state] (C3neg3) at (22,-1) {};
\node[state] (C3neg4) at (23,-1) {};

\node[state] (C3end) at (24,0) {};
\node[state] (t) at (25,0) {$t$};
\draw [ ->] (C3end) edge (t);

\draw [ ->] (3con) edge (C1start); 
\draw [ ->] (C1start) edge (C1pos1); 
\draw [ ->] (C1start) edge (C1neg1);
\draw [ ->] (C1pos1) edge (C1pos2); 
\draw [ ->] (C1pos2) edge (C1pos3); 
\draw [ ->] (C1pos3) edge (C1pos4); 
\draw [ ->] (C1pos4) edge (C1end); 
\draw [ ->] (C1neg1) edge (C1neg2); 
\draw [ ->] (C1neg2) edge (C1neg3); 
\draw [ ->] (C1neg3) edge (C1neg4); 
\draw [ ->] (C1neg4) edge (C1end);

\draw [ ->] (C1end) edge (C2start); 
\draw [ ->] (C2start) edge (C2pos1); 
\draw [ ->] (C2start) edge (C2neg1);
\draw [ ->] (C2pos1) edge (C2pos2); 
\draw [ ->] (C2pos2) edge (C2pos3); 
\draw [ ->] (C2pos3) edge (C2pos4); 
\draw [ ->] (C2pos4) edge (C2end); 

\draw [ ->] (C2neg1) edge (C2neg2); 
\draw [ ->] (C2neg2) edge (C2neg3); 
\draw [ ->] (C2neg3) edge (C2neg4); 
\draw [ ->] (C2neg4) edge (C2end); 

\draw [ ->] (C2end) edge (C3start); 
\draw [ ->] (C3start) edge (C3pos1); 
\draw [ ->] (C3start) edge (C3neg1);
\draw [ ->] (C3pos1) edge (C3pos2); 
\draw [ ->] (C3pos2) edge (C3pos3); 
\draw [ ->] (C3pos3) edge (C3pos4); 
\draw [ ->] (C3pos4) edge (C3end); 

\draw [ ->] (C3neg1) edge (C3neg2); 
\draw [ ->] (C3neg2) edge (C3neg3); 
\draw [ ->] (C3neg3) edge (C3neg4); 
\draw [ ->] (C3neg4) edge (C3end);

\draw [thick, ->] (1pos) edge [bend right=70] (C1neg1);
\draw [thick, ->] (1neg) edge [bend left=70] (C1pos1);

\draw [thick, ->] (2pos) edge [bend right=70] (C2neg1);
\draw [thick, ->] (2neg) edge [bend left=70] (C2pos1);

\draw [thick, ->] (3pos) edge [bend right=70] (C3neg1);
\draw [thick, ->] (3neg) edge [bend left=70] (C3pos1);

\draw [thick, ->] (C1pos2) edge [bend left=70] (t);
\draw [thick, ->] (C1neg2) edge [bend right=70] (t);

\draw [thick, ->] (C2pos2) edge [bend left=70] (t);
\draw [thick, ->] (C2neg2) edge [bend right=70] (t);

\draw [thick, ->] (C3pos2) edge [bend left=60] (t);
\draw [thick, ->] (C3neg2) edge [bend right=60] (t);

\end{tikzpicture}
}
\caption{The additional edges in the graph $(G_2)$ that model the relationship between $S(x)$ guarantee the consistency of literal assignment and literal assignment in each clause are thick. Each thick edge in the figure corresponds to a chain of edges in the graph.}
\label{fig:graph_G2}
\end{figure}
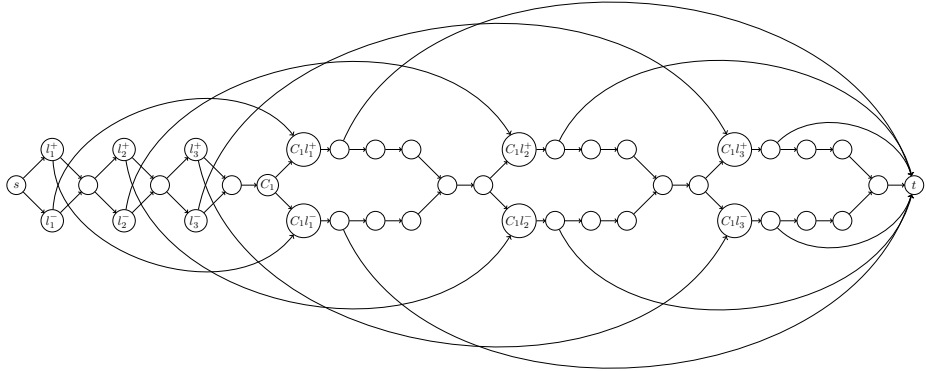

We introduce chains of edges that connect the first part of $G_1$ with the second part of $G_1$ as shown in Figure~\ref{fig:graph_G2}. Note that the length of each chain can be chosen in such a way that all $s-t$ paths have the same length. Assume that path $x$ represents an inconsistent assignment for $l_1$, e.g., let $x$ contain node $l_1^+$ and $C_1l_1^-$. We claim that in this case a path $y$ exists with $x^Ty=0$. Consider the path $y$ that contains from $G_1$ only the nodes $s,l_1^-$,$C_1l_1^+$, the successor node of $C_1l_1^+$ and $t$. This path has no arc in common with $x$, hence,  $x^Ty=0$. This relation holds analogously for the other literals $l_2$ and $l_3$. If only a single inconsistent assignment is made, there exists a path $y$ with $x^Ty=0$. On the other hand, if $x$ represents a consistent assignment, all paths $y$ in $G_1$ and $G_2$ have at least one edge in common with $x$.

Next we introduce part $G_3$ that models the relationship between $S(x)$ and the correct clause assignment. The additional chains of edges are shown in Figure~\ref{fig:graph_G3}. Again the length of each of these chains can be chosen in such a way that all $s-t$ paths have the same length.

\begin{figure}[h]
\centering
\resizebox{\textwidth}{!}{%
\begin{tikzpicture}
\tikzstyle{every state}=[circle,fill=white!25,minimum size=15pt,inner sep=1pt];
\node[state] (0) at (0,0) {$s$};
\node[state] (1pos) at (1,1) {$l_1^+$};
\node[state] (1neg) at (1,-1) {$l_1^-$};
\node[state] (1con) at (2,0) {};

\node[state] (2pos) at (3,1) {$l_2^+$};
\node[state] (2neg) at (3,-1) {$l_2^-$};
\node[state] (2con) at (4,0) {};

\node[state] (3pos) at (5,1) {$l_3^+$};
\node[state] (3neg) at (5,-1) {$l_3^-$};
\node[state] (3con) at (6,0) {};

\draw [ ->] (0) edge (1pos); 
\draw [ ->] (0) edge (1neg); 
\draw [ ->] (1pos) edge (1con); 
\draw [ ->] (1neg) edge (1con); 

\draw [ ->] (1con) edge (2pos); 
\draw [ ->] (1con) edge (2neg); 
\draw [ ->] (2pos) edge (2con); 
\draw [ ->] (2neg) edge (2con); 

\draw [ ->] (2con) edge (3pos); 
\draw [ ->] (2con) edge (3neg); 
\draw [ ->] (3pos) edge (3con); 
\draw [ ->] (3neg) edge (3con); 

\node[state] (C1start) at (7,0) {$C_1$};

\node[state] (C1pos1) at (8,1) {$C_1l_1^+$};
\node[state] (C1pos2) at (9,1) {};
\node[state] (C1pos3) at (10,1) {};
\node[state] (C1pos4) at (11,1) {};

\node[state] (C1neg1) at (8,-1) {$C_1l_1^-$};
\node[state] (C1neg2) at (9,-1) {};
\node[state] (C1neg3) at (10,-1) {};
\node[state] (C1neg4) at (11,-1) {};

\node[state] (C1end) at (12,0) {};

\node[state] (C2start) at (13,0) {};

\node[state] (C2pos1) at (14,1) {$C_1l_2^+$};
\node[state] (C2pos2) at (15,1) {};
\node[state] (C2pos3) at (16,1) {};
\node[state] (C2pos4) at (17,1) {};

\node[state] (C2neg1) at (14,-1) {$C_1l_2^-$};
\node[state] (C2neg2) at (15,-1) {};
\node[state] (C2neg3) at (16,-1) {};
\node[state] (C2neg4) at (17,-1) {};

\node[state] (C2end) at (18,0) {};

\node[state] (C3start) at (19,0) {};

\node[state] (C3pos1) at (20,1) {$C_1l_3^+$};
\node[state] (C3pos2) at (21,1) {};
\node[state] (C3pos3) at (22,1) {};
\node[state] (C3pos4) at (23,1) {};

\node[state] (C3neg1) at (20,-1) {$C_1l_3^-$};
\node[state] (C3neg2) at (21,-1) {};
\node[state] (C3neg3) at (22,-1) {};
\node[state] (C3neg4) at (23,-1) {};

\node[state] (C3end) at (24,0) {};
\node[state] (t) at (25,0) {$t$};
\draw [ ->] (C3end) edge (t); 

\draw [ ->] (3con) edge (C1start); 
\draw [ ->] (C1start) edge (C1pos1); 
\draw [ ->] (C1start) edge (C1neg1);
\draw [ ->] (C1pos1) edge (C1pos2); 
\draw [ ->] (C1pos2) edge (C1pos3); 
\draw [ ->] (C1pos3) edge (C1pos4); 
\draw [ ->] (C1pos4) edge (C1end); 
\draw [ ->] (C1neg1) edge (C1neg2); 
\draw [ ->] (C1neg2) edge (C1neg3); 
\draw [ ->] (C1neg3) edge (C1neg4); 
\draw [ ->] (C1neg4) edge (C1end);

\draw [ ->] (C1end) edge (C2start); 
\draw [ ->] (C2start) edge (C2pos1); 
\draw [ ->] (C2start) edge (C2neg1);
\draw [ ->] (C2pos1) edge (C2pos2); 
\draw [ ->] (C2pos2) edge (C2pos3); 
\draw [ ->] (C2pos3) edge (C2pos4); 
\draw [ ->] (C2pos4) edge (C2end); 

\draw [ ->] (C2neg1) edge (C2neg2); 
\draw [ ->] (C2neg2) edge (C2neg3); 
\draw [ ->] (C2neg3) edge (C2neg4); 
\draw [ ->] (C2neg4) edge (C2end); 

\draw [ ->] (C2end) edge (C3start); 
\draw [ ->] (C3start) edge (C3pos1); 
\draw [ ->] (C3start) edge (C3neg1);
\draw [ ->] (C3pos1) edge (C3pos2); 
\draw [ ->] (C3pos2) edge (C3pos3); 
\draw [ ->] (C3pos3) edge (C3pos4); 
\draw [ ->] (C3pos4) edge (C3end); 

\draw [ ->] (C3neg1) edge (C3neg2); 
\draw [ ->] (C3neg2) edge (C3neg3); 
\draw [ ->] (C3neg3) edge (C3neg4); 
\draw [ ->] (C3neg4) edge (C3end);

\draw [thick, ->] (1pos) edge [bend right=70] (C1neg1);
\draw [thick, ->] (1neg) edge [bend left=70] (C1pos1);

\draw [thick, ->] (2pos) edge [bend right=70] (C2neg1);
\draw [thick, ->] (2neg) edge [bend left=70] (C2pos1);

\draw [thick, ->] (3pos) edge [bend right=70] (C3neg1);
\draw [thick, ->] (3neg) edge [bend left=70] (C3pos1);

\draw [thick, ->] (C1pos2) edge [bend left=70] (t);
\draw [thick, ->] (C1neg2) edge [bend right=70] (t);

\draw [thick, ->] (C2pos2) edge [bend left=70] (t);
\draw [thick, ->] (C2neg2) edge [bend right=70] (t);

\draw [thick, ->] (C3pos2) edge [bend left=60] (t);
\draw [thick, ->] (C3neg2) edge [bend right=60] (t);

\draw[thick, dotted, ->] (0) edge [bend left=30] (C1pos3);
\draw[thick, dotted, ->] (C1pos4) edge (C2neg3);
\draw[thick, dotted, ->] (C2neg4) edge (C3pos3);
\draw[thick, dotted, ->] (C3pos4) edge (t);

\end{tikzpicture}
}
\caption{The additional edges in the graph ($G_3$) that model the relationship between $S(x)$ and the correct clause assignment are dotted. Each dotted edge in the figure corresponds to a chain of edges in the graph.}
\label{fig:graph_G3}
\end{figure}
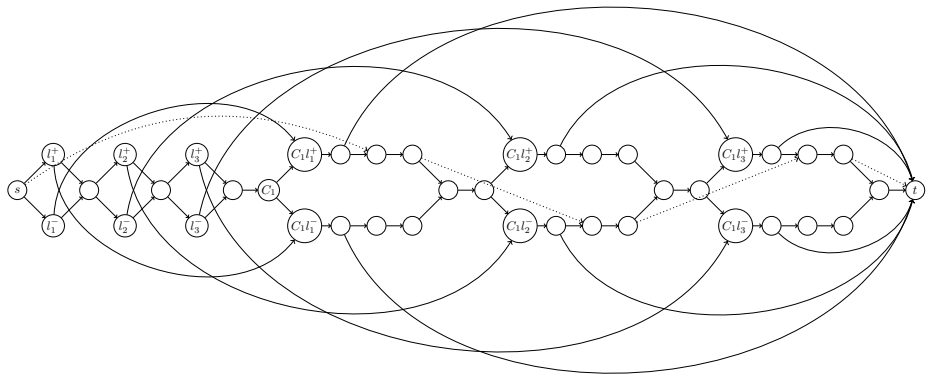

Assume that clause $C_1$ is not satisfied by the represented literal assignment, i.e., path $x$ contains $C_1l_1^-,C_1l_2^+,$ and $C_1l_3^-$. It is obvious that the path $y$ that contains all three of the dotted chains has no edge in common with $x$ and, hence, $S(x)=0$. Conversely, if only one literal is assigned such that $C_1$ is fulfilled, this path shares at least one arc with $x$. 

We now show that if path $x$ represents a consistent literal and clause assignment, then $S(x)\geq 1$, i.e., for every path $y$ it holds that $x^Ty\geq 1$, if and only if all clauses are fulfilled.

Assume that $x$ represents a literal assignment that fulfills all clauses. For the sake of contradiction assume that a path $y$ exists that has no edge in common with $x$. It is an easy observation that $y$ contains either one of the thick or one of the dotted edges as, otherwise, it must contain the edge that leads to vertex $C_1$ which is also contained in $x$. If $y$ contains one dotted arc that leads to some clause it must also contain the other dotted arcs that belong to this clause as $x$ contains the edges that connect the three parts of this clause. Hence, the argument from above is valid. If $y$ contains one of the thick arcs this arc must correspond to a conflicting literal assignment (with respect to the assignment of $x$), as every thick arc is connected to the contradicting assignment in the clause. The next edge of this path is contained in $x$, as $x$ represents a consistent literal assignment. 

On the other hand, if $x$ represents a literal assignment that violates at least one clause, there exists obviously a path $y$ using the corresponding dotted edges with $x^Ty=0$.

\medskip

To conclude the proof, we have to show the two open claims.

\medskip

\noindent
\textbf{Proof of Claim 1} \\
Let $x$ be an arbitrary path in $G$. Observe that not both nodes $l_1^+$ and $l_1^-$ can be contained in $x$. With out loss of generality let $l_1^+$  be not contained in $x$. We construct a path $y$ which shares at most one edge with $x$. Denote by $v$ the successor node of $C_1l_1^-$. Path $y$ starts with edge $(s,l_1^+)$ next it uses the chain of edges from $l_1^+$ to $C_1l_1^-$ and edge $(C_1l_1^-,v)$. If $x$ contains the chain of edges from $v$ to $t$, which are part of $G_2$, we continue path $y$ by an arbitrary path from $v$ to $t$ contained in $G_1$. In the other case, where the chain of edges from $v$ to $t$ is not contained in $x$, we continue path $y$ simply with this chain. The constructed path $y$ shares at most the edge $(C_1l_1^-,v)$ with $x$. This proves Claim~1. \\

\noindent
\textbf{Proof of Claim 2} \\
Let $x$ be an arbitrary path in $G$ with $S(x)=1$. We claim that $x$ must fulfill the following properties: $x$ must contain node $C_1$ and $x$ must contain at least one of the nodes $C_il_k^+$ or $C_il_k^-$. 

For the sake of contradiction assume first that $x$ does not contain $C_1$. Then there exists a path $y$ from $s$ to $C_1$ not sharing any edge with $x$. This path can easily be extended to an $s-t$ path sharing no edge with $x$ by either adding the path from $C_1$ to $C_1l_1^+$ to $t$ (using a chain of edges from $G_2$) or the path from $C_1$ to $C_1l_1^-$ to $t$ (using a chain of edges from $G_2$) respectively.

For the sake of contradiction assume without loss of generality that $C_1l_1^+$ and $C_1l_1^-$ are not contained in $x$. Again we construct a path $y$ that shares no edge with $x$. With out loss of generality assume that $l_1^+$ is not contained in $x$. Path $y$ starts with edge $(s,l_1^+)$, followed by the chain of edges from $G_2$ going from $l_1^+$ to $C_1l_1^-$, the edge $(C_1l_1^-,v)$ and the chain of edges from $G_2$ from $v$ to $t$. Note that $y$ shares no edge with $x$.

Note that the first possible node a path $x$ fulfilling both of these properties can leave $G_1$ is at the third part of the last clause node. Denote by $u$ the last node of $x$ contained in $G_1$ (except for $t$). Note that there is only a single path $\tilde{x}$ from $u$ to $t$ in $G_1$. Consider the following path $x'$. The first part from $s$ to $u$ coincides with $x$. The second part is equal to $\tilde{x}$. Note that $x'$ is contained in $G_1$ and shares edges with all paths $y$ that share edges with $x$. Hence, $S(x')\geq S(x)$. This concludes the proof of Claim~2.

Note that the presented reduction uses a $3$-SAT instance which consists of a single clause. The presented ideas generalize straightforward to the case of arbitrary $3$-SAT instances. To introduce an additional literal $l_n$, the first part of the graph $G_1$ is extended by $l_n^+$ and $l_n^-$. The gadget representing an additional clause $C_m$ has exactly the same structure as the gadget for $C_1$. The corresponding gadget of $C_m$ is put at the end of the graph.

\end{proof}

Note that the instance used for the reduction in the proof of Theorem~\ref{thm_sp_sol_ell} is strongly restricted: All edges have the same cost structure, the uncertainty set is a perfect ball and all $s-t$ paths contain the same number of edges. The complexity results of this section are summarized in Table~\ref{tab:sp}.

\begin{table}[h]
\centering
\begin{tabular}{|c|c|c|c|c|}
\hline
& Interval & Axis-Parallel Ellipsoid & General Ellipsoid & Finite  \\ \hline
Eval & \textsf{P} & \textsf{NPC} (Thm.~\ref{thm_sp_eval_ell}) & \textsf{NPC} (Thm.~\ref{thm_sp_eval_ell}) & \textsf{P} \\ 
Solve & \textsf{NPC} & \textsf{NPH} (Thm.~\ref{thm_sp_sol_ell}) & \textsf{NPH} (Thm.~\ref{thm_sp_sol_ell}) & \textsf{NPC} \\ \hline
\end{tabular}
\caption{Overview of the different complexity results of the minmax regret shortest path problem.}\label{tab:sp}
\end{table}

\section{Solution Approaches}
\label{sec:exact}

In this section we discuss solution approaches for the minmax regret problem with ellipsoidal uncertainty sets. We begin with briefly revisiting the scenario relaxation procedure for interval uncertainty in Section~\ref{sec:relint}, before introducing exact solution approaches for ellipsoidal sets in Section~\ref{sec:relell}.

\subsection{Scenario Relaxation for Interval Sets}
\label{sec:relint}

For combinatorial minmax regret problems with interval uncertainty sets, one of the most frequently used solution method is to generate a finite set of scenarios iteratively (see \cite{Aissi2009}). There are (at least) two ways to do so.
We briefly explain them in the following.

A general minmax regret problem of the form
\[ \min_{x\in\X} \max_{c\in\cU} \left( c^Tx - opt(c) \right) \]
can be rewritten as:
\begin{align*}
\min\ & z \\
\text{s.t. } & z \ge c^Tx - c^T y &\forall c\in\cU, y\in \X\\
& x \in \X
\end{align*}
In case of an interval uncertainty set, these are infinitely many constraints. Even restricting ourselves to extreme points of the uncertainty set, there are still exponentially many. For this reason, we generate them iteratively during the solution process.

\noindent
Let us consider the constraints
\[ \Big( z \ge c^Tx - c^T y \ \forall y\in \X \Big) \qquad \forall c\in\cU \]
with $\cU = \bigtimes_{i\in[n]} [\underline{c}_i, \oc_i]$. If we fix some $c\in\cU$, we can read this as
\[ z \ge \max_{y\in\X} \left( c^Tx - c^Ty \right) \]
which is equivalent to 
\begin{equation}
z \ge c^Tx - opt(c). \label{cut-1}
\end{equation}
That is, we can iteratively generate scenarios $c\in\cU$ and add constraints of the form \eqref{cut-1} to solve the robust problem. To find the the next $c\in\cU$ in each iteration (that is, a maximizer of the right-hand side of \eqref{cut-1}), one simply uses $c^*(x)$, with $c^*(x)_i := \underline{c}_i + (\overline{c}_i-\underline{c}_i)x_i$. 
To find $opt(c^*(x))$, a problem of the nominal type needs to be solved. We refer to constraints of this kind as type 1 cuts.

\noindent
Analogously, we can consider constraints of the form
\[ \Big( z \ge c^Tx - c^T y \ \forall c\in\cU \Big) \qquad \forall y\in \X. \] 
That is, for fixed $y\in\X$, let us consider 
\[ z \ge \max_{c\in\cU} \left( c^Tx - c^T y \right) \]
in more detail. This is equivalent to setting
\begin{equation}
z \ge c^*(x)^Tx - c^*(x)^Ty. \label{cut-2}
\end{equation}
It is then possible to rewrite $c^*(x)$ such that this becomes a linear integer program. We refer to constraints of this kind as type 2 cuts. To find the next such cut, we need to solve a nominal problem with $c^*(x)$, just like for type 1. 

Note that type 2 cuts are more ''flexible'' in the sense that they only fix a solution $y$, and use the worst-case scenario depending on $x$. For type 1 cuts, both scenario $c$ and solution $y$ are fixed. For this reason, it can be shown that type 2 cuts are more efficient (tighter) than type 1 cuts \cite{Aissi2009}.

\subsection{Solution Approaches for Ellipsoidal Sets}
\label{sec:relell}

We now consider minmax regret problems with general ellipsoidal uncertainty sets $\cU = \{ \hat{c} + C\xi : \Vert \xi \Vert_2 \le 1 \}$. As in the case of interval uncertainty sets, we consider two ways to reformulate the constraints
\[ z \ge c^Tx - c^T y \qquad \forall c\in\cU, y\in \X \]
First, let us fix $c\in\cU$. Then, just as before, the constraints become equivalent to
\[ z\ge \max_{y\in\X} \left( c^Tx - c^Ty \right) \qquad \iff \qquad z \ge c^Tx - opt(c). \] 
However, generating the next such constraint for a given $x\in\X$ is more complex. We need to solve the problem of finding the largest such cut, that is,
\[ \max_{c\in\cU} \left( c^Tx - opt(c) \right). \]
This is equivalent to:
\begin{align*}
\max \ & c^Tx - c^T y \\
\text{s.t. } & c = \hat{c} + C\xi \\
& \Vert \xi \Vert_2 \le 1 \\
& y\in \X
\end{align*}
Using Lemma~\ref{lem1a}, we find that this problem is equivalent to 
\begin{align}
\max\ &\hat{c}^T(x-y) + z \nonumber\\
\text{s.t. } & z^2 \le \Vert C^T (x-y) \Vert_2^2 \tag{SUB}\\
& y\in \X, z \ge 0.\nonumber
\end{align}
To solve problem (SUB), we consider two linearizations of the right-hand side. In our first approach, we use that $x_i = x_i^2$ for binary variables $x_i$ and find that
\[ \Vert C^T (x-y) \Vert_2^2 = \sum_{i\in[n]} \sum_{j\in[n]} \left( C^2_{ji} (x_j - 2x_j y_j + y_j) +  \sum_{k<j} 2C_{ji}C_{ki}(x_j-y_j)(x_k-y_k) \right) \]
To linearize products of the form $y_jy_k$, we introduce new binary variables $\alpha_{jk}$ with
\[ y_j + y_k \le 1 + \alpha_{jk} \qquad \text{and} \qquad 2\alpha_{jk} \le y_j + y_k. \]
Using this linearization of the right-hand side in (SUB), we arrive at a convex quadratic integer program.

As a second approach, we rewrite the constraint as
\[ \Vert C^T (x-y) \Vert_2^2 = v^T Q v = \sum_{i\in[n]} v_i a_i(v) \]
with $v_i := x_i - y_i$, $Q := CC^T$ and $a_j(v) := (Qv)_j = \sum_{i\in[n]} q_{ji} v_i$. We introduce new variables $h_j := v_j a_j(v)$ and linearize them using the following constraints. For any $j\in[n]$ with $v_j \in \{0,1\}$ (i.e., $x_j = 1$) we set
\[ h_j \le \sum_{i\in[n]} q_{ji} v_i + M^-_j (1-v_j) \qquad \text{ and } \qquad h_j \le M^+_j v_j. \]
For any $j\in [n]$ with $v_j \in \{-1,0\}$ (i.e., $x_j = 0$), we use instead
\[ h_j \le - \sum_{i\in[n]} q_{ji} v_i + M^+_j (1 + v_j) \qquad \text{ and } \qquad h_j \le -M^-_j v_j. \]
The constants $M^+_j$ and $M^-_j$ are chosen such that $M^+_j \ge \max_v \sum_{i\in[n]} q_{ji} v_i$ and $M^-_{ij} \ge - \min_v \sum_{i\in[n]} q_{ji} v_i$. To this end, we set $M^+_j := \sum_{i\in[n]} q_{ji} x_i$ and $M^-_j := \sum_{i\in[n]} q_{ji} (1-x_i)$ as the smallest possible such constants.

Note that the second linearization requires less additional variables (linearly instead of quadratically many), but is numerically less stable due to the ''big-$M$'' constraints.

Solving (SUB) we find $y^*$,
and the corresponding $c^*$ is given by $\hc + C\xi^*$ with $\xi^* = C^T(x-y^*)/\Vert C^T (x-y^*) \Vert_2$.

As for interval uncertainty sets, we refer to this procedure as type 1 cuts.

\bigskip

For the second type of cuts, we fix $y\in\X$, in which case our constraints become
\[ z \ge c^Tx - c^T y \qquad \forall c\in\cU \]
which is a ''classic'' robust optimization constraint, i.e., using Lemma~\ref{lem1a} it can be reformulated to
\[ z \ge \hat{c}^T (x-y) + \Vert C^T (x-y) \Vert_2. \]
This is again a conic quadratic constraint. To generate new cuts of this form, we maximize the right-hand-side in $y$, which is the same subproblem as described in (SUB).

To summarize, both approaches need to solve the same subproblem to generate new cuts. Using cuts of type 1 amounts to master problems that are integer linear, while cuts of type 2 amount to master problems that are
second order cone integer.
In principle, master problems for type 2 are therefore harder to solve. However, they have the advantage that they give a tighter formulation.

\begin{theorem}
Cuts of type 2 are tighter than cuts of type 1.
\end{theorem}
\begin{proof}
Let some $x\in\X$ be fixed, and let $c$ and $y$ be generated from the subproblem (SUB). Then we have
\begin{align*}
c^T x - opt(c) &= c^Tx - c^Ty\\
& \le \max_{c'\in\cU} \left( {c'}^Tx - {c'}^Ty \right) = \hat{c}^T(x-y) + \Vert C^T (x-y) \Vert_2
\end{align*}
\end{proof}

We conclude this section by considering an approximation algorithm. As one can easily see, $\cU$ is symmetric with respect to $\hat{c}$. Using Property 3.3 from \cite{Conde2012452}, we get the following result.

\begin{theorem}
The midpoint solution
\[ \hat{x} \in \arg \min \{ \hat{c}^T x : x\in\X \} \]
is a 2-approximation for the minmax regret problem with ellipsoidal uncertainty set.
\end{theorem}

\section{Computational Experiments}
\label{sec:experiments}

The purpose of these experiments it to compare the performance of type 1 and type 2 cuts for general ellipsoidal uncertainty sets, using one of the two linearizations for problem (SUB). To this end, we use both unconstrained and shortest path problems as a testbed.

\subsection{Setup}

We generate uncertain unconstrained problems of the form
\[ \min \left\{ c^T x : x\in\{0,1\}^n \right\} \]
by creating random ellipsoidal uncertainty sets $\cU$. For all instances, we generate $\hc_i \in \{-100,\ldots,100\}$ and $C_{ii} \in \{50,\ldots,150\}$. Additionally, non-diagonal entries of $C$ are generated in three different ways:
\begin{itemize}
\item Sets with small deviation, where $C_{ij} \in \{1,\ldots,50\}$

\item Sets with medium deviation, where $C_{ij} \in \{1,\ldots,50\}$ with a probability of $75\%$, and in $C_{ij} \in \{50,\ldots,200\}$ with $25\%$.

\item Sets with large deviation, where $C_{ij} \in \{50,\ldots,200\}$. 
\end{itemize}
Parameters were always generated uniformly at random from the respective sets of possible outcomes. Each non-diagonal entry is generated with a certain probability $p\in\{5\%,15\%,25\%\}$. 
For each number of items $n$ in $\mathcal{N}=\{10 + 20N : N\in\{0,\ldots,7\}\}$ we therefore generated nine instance sets, which which we denote as $\cI^{p,y}_n$ with $n$ items and $y\in\{s,m,l\}$ for small, medium, and large deviation, respectively. We abbreviate $\cI^s$, $\cI^m$, $\cI^l$ and $\cI^5$, $\cI^{15}$, $\cI^{25}$ to denote all instances of the respective type (i.e., $\cI^m$ denotes all instances with medium deviation, and $\cI^5$ denotes all instances where non-diagonal entries are generated with $5\%$ probability).
For each instance set, we generated 10 instances, which means a total of 720 instances were considered.

\medskip

Additionally, we generated a second set of test instances for shortest path problems.
All parameters are chosen in the same way as for the unconstrained problems. The graphs we consider are layered graphs with 4 nodes per layer, and $n$ layers with $n\in\{2,\ldots,9\}$. Between two layers, all possible forward edges were generated.
We denote these instances as $\cJ^{p,y}_n$, with the same abbreviations as for $\cI$. For each instance set, 10 instances were generated (720 instances in total).

We use the two scenario relaxation procedures described in Section~\ref{sec:exact}. In the following, we denote the solution approach that uses type 1 cuts of the form
\[ z \ge c^Tx - opt(c) \]
as C1, and the approach based on type 2 cuts of the form
\[ z \ge \hat{c}^T (x-y) + \Vert C^T (x-y) \Vert_2 \]
as C2. Recall that C1 generates master problems that are likely to be easier to solve, while C2 has tighter bounds and might need less iterations. Depending on how the subproblem (SUB) is linearized, we append either ''-A'' (for the first linearization with quadratically many variables) or ''-B'' (for the second linearization with linearly many variables) to the name of the method.

We used CPLEX v.12.6 \cite{cplex} to solve all linear and quadratic integer programs on a computer with a 16-core Intel Xeon E5-2670 processor, running at 2.60 GHz with 20MB cache, and Ubuntu 12.04. Processes were pinned to one core. A time limit of 900 seconds was used per method and instance.

\subsection{Experiment 1: Unconstrained Problems}

Figure~\ref{fig:ex0all} shows the resulting performance profile over all 720 unconstrained instances, i.e., at every time step, we plot how many instances have been solved to optimality. Plotted in black is C1, while C2 is in blue. Method A linearization of sub is a full line, and method B linearization is a dashed line. In Figure~\ref{fig:ex0}, the performance is shown over different instance classes.

\begin{figure}[htbp]
\centering
\includegraphics[width=\textwidth]{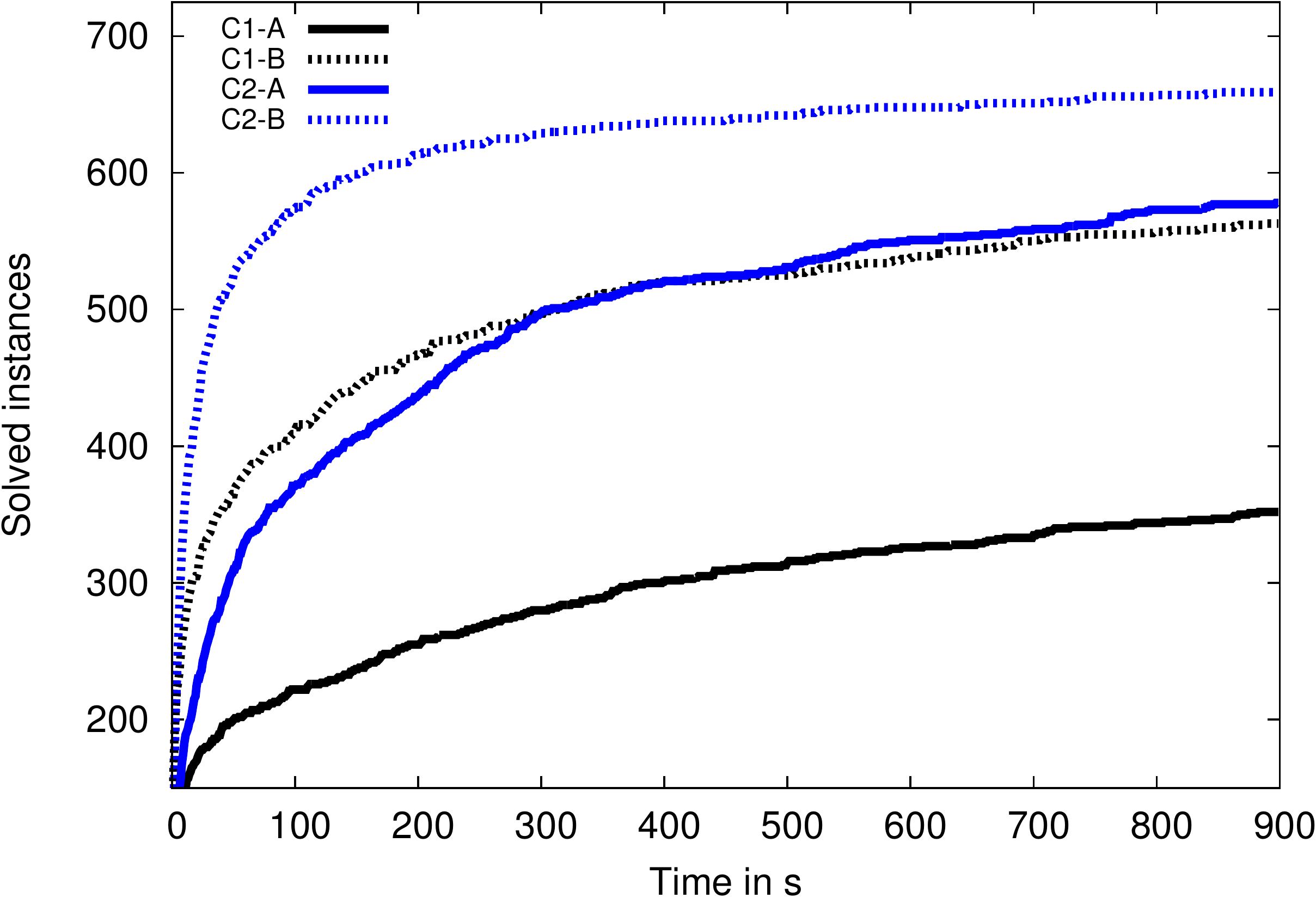}
\caption{Performance profile for unconstrained problems, all instances.}\label{fig:ex0all}
\end{figure}

\begin{figure}[htbp]
\centering
\subfigure[Instances $\cI^s$.\label{plot-small-0}]{\includegraphics[width=.5\textwidth]{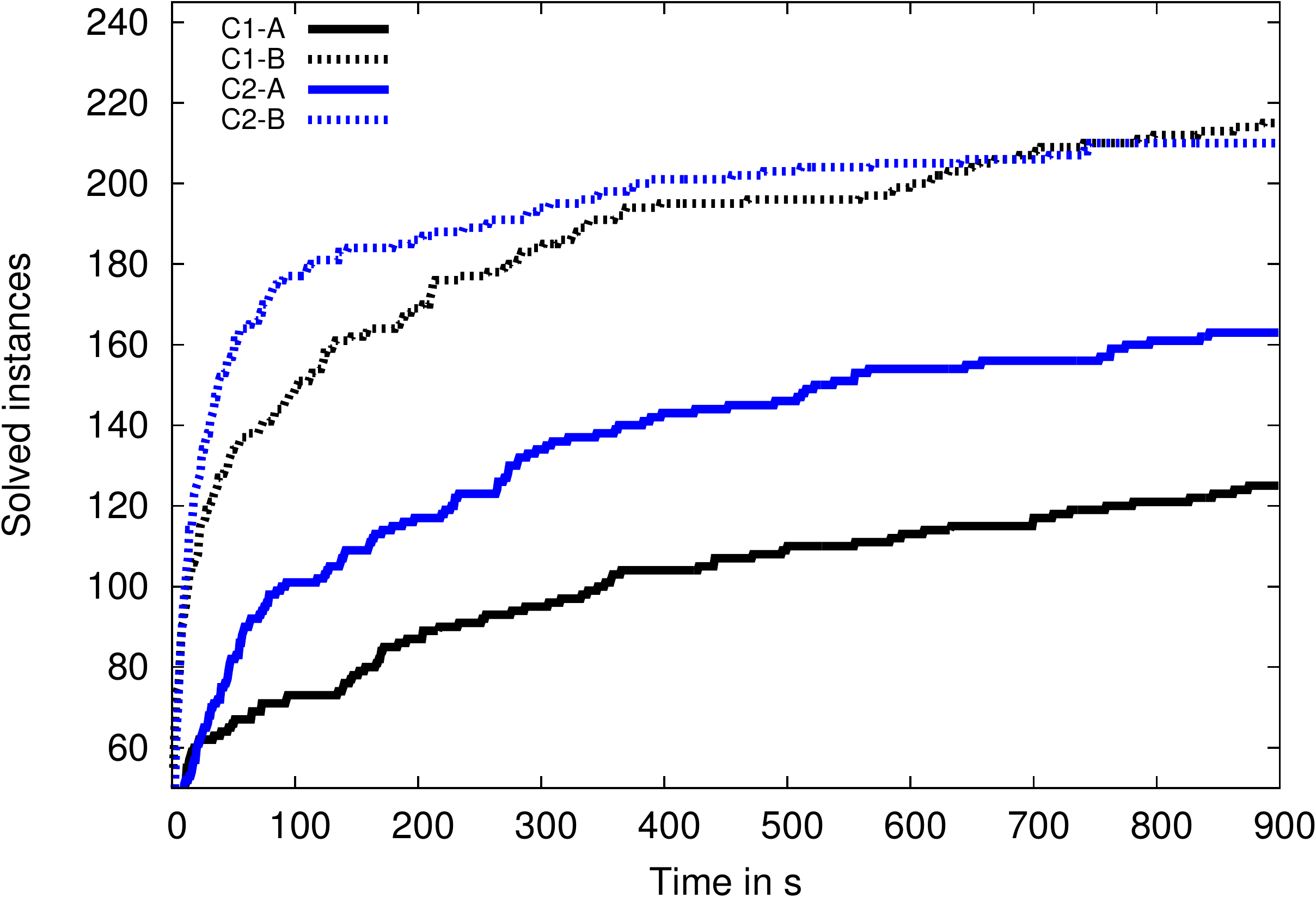}}%
\subfigure[Instances $\cI^m$.\label{plot-medium-0}]{\includegraphics[width=.5\textwidth]{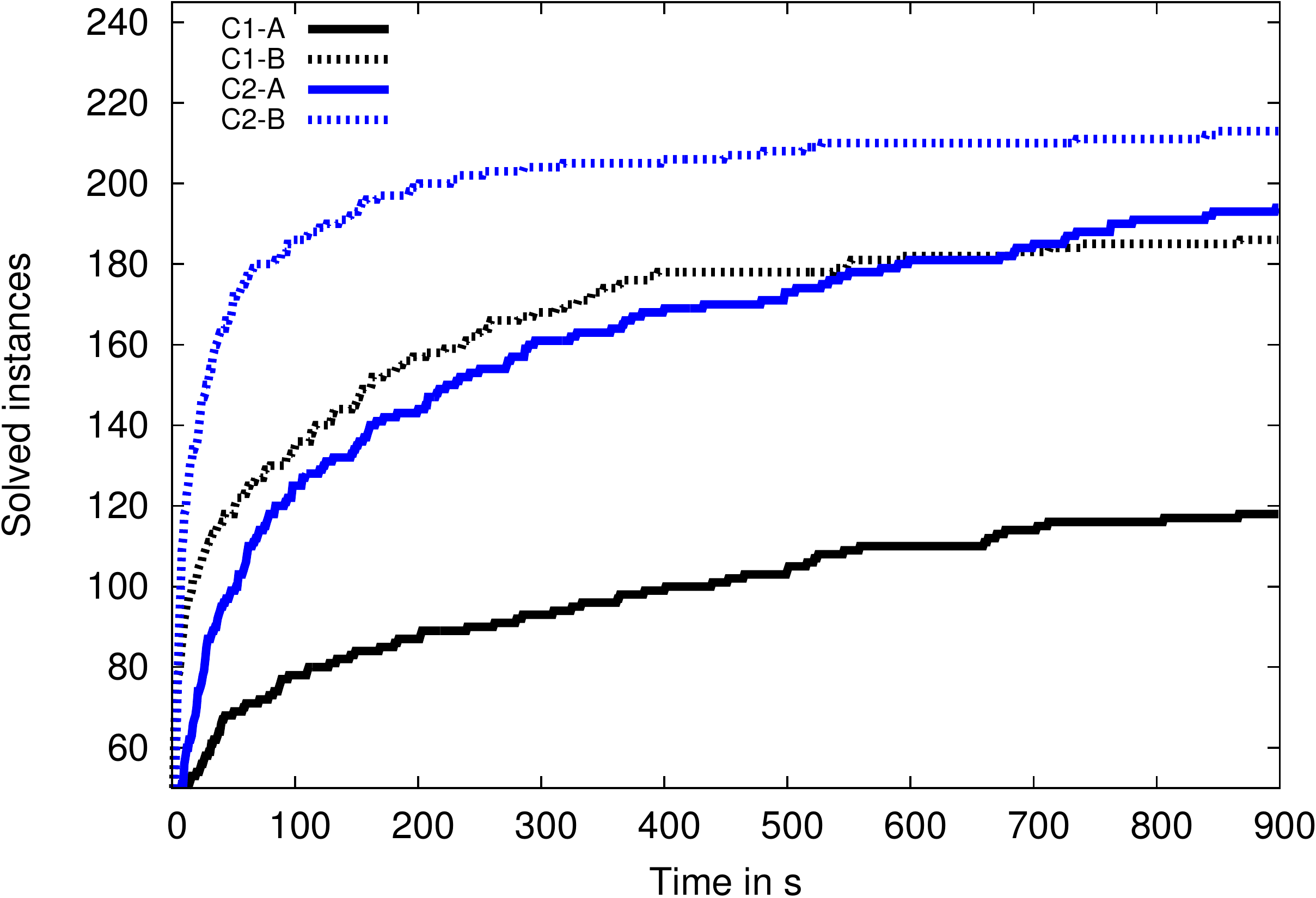}}

\subfigure[Instances $\cI^l$.\label{plot-large-0}]{\includegraphics[width=.5\textwidth]{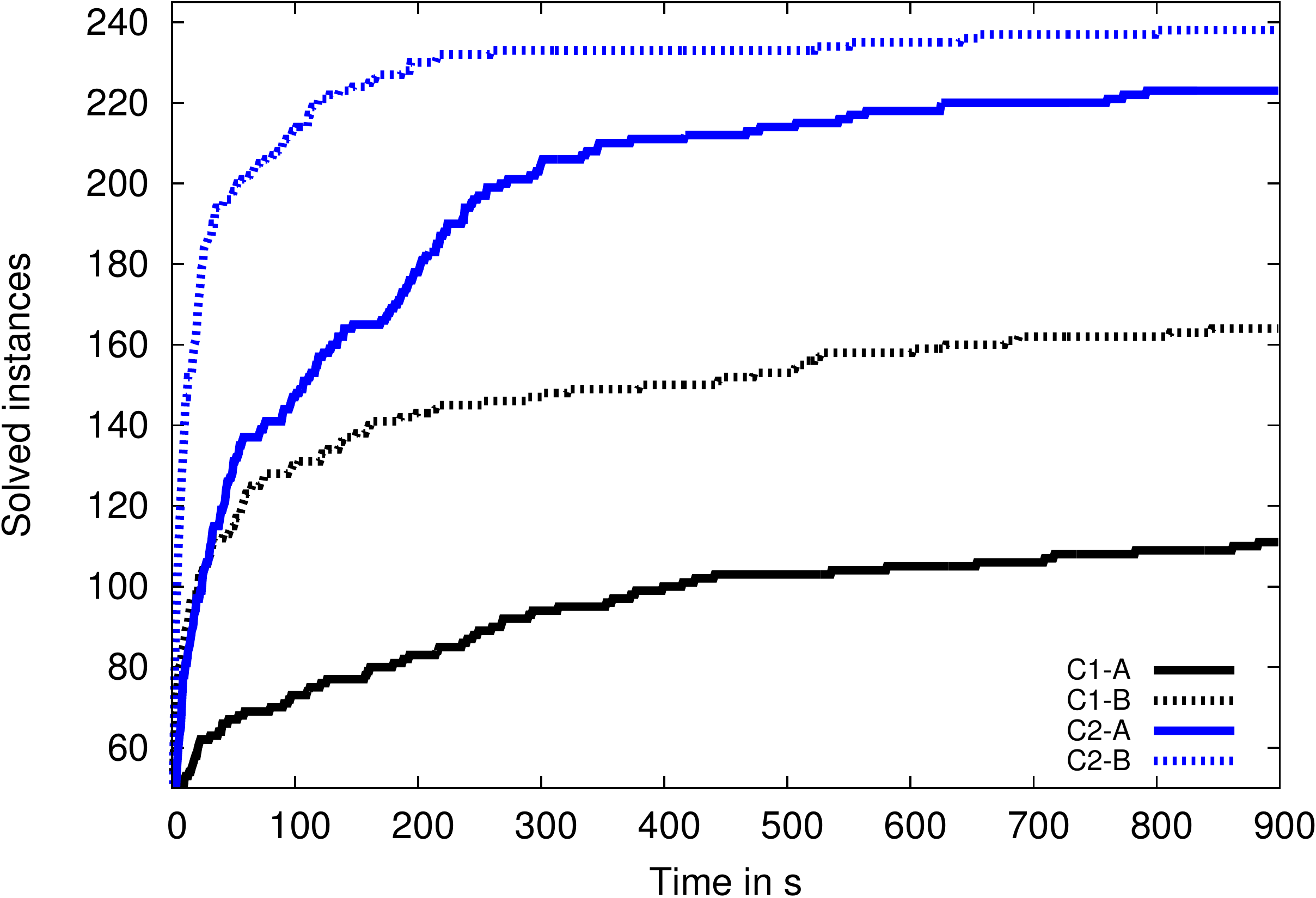}}%
\subfigure[Instances $\cI^5$.\label{plot-5-0}]{\includegraphics[width=.5\textwidth]{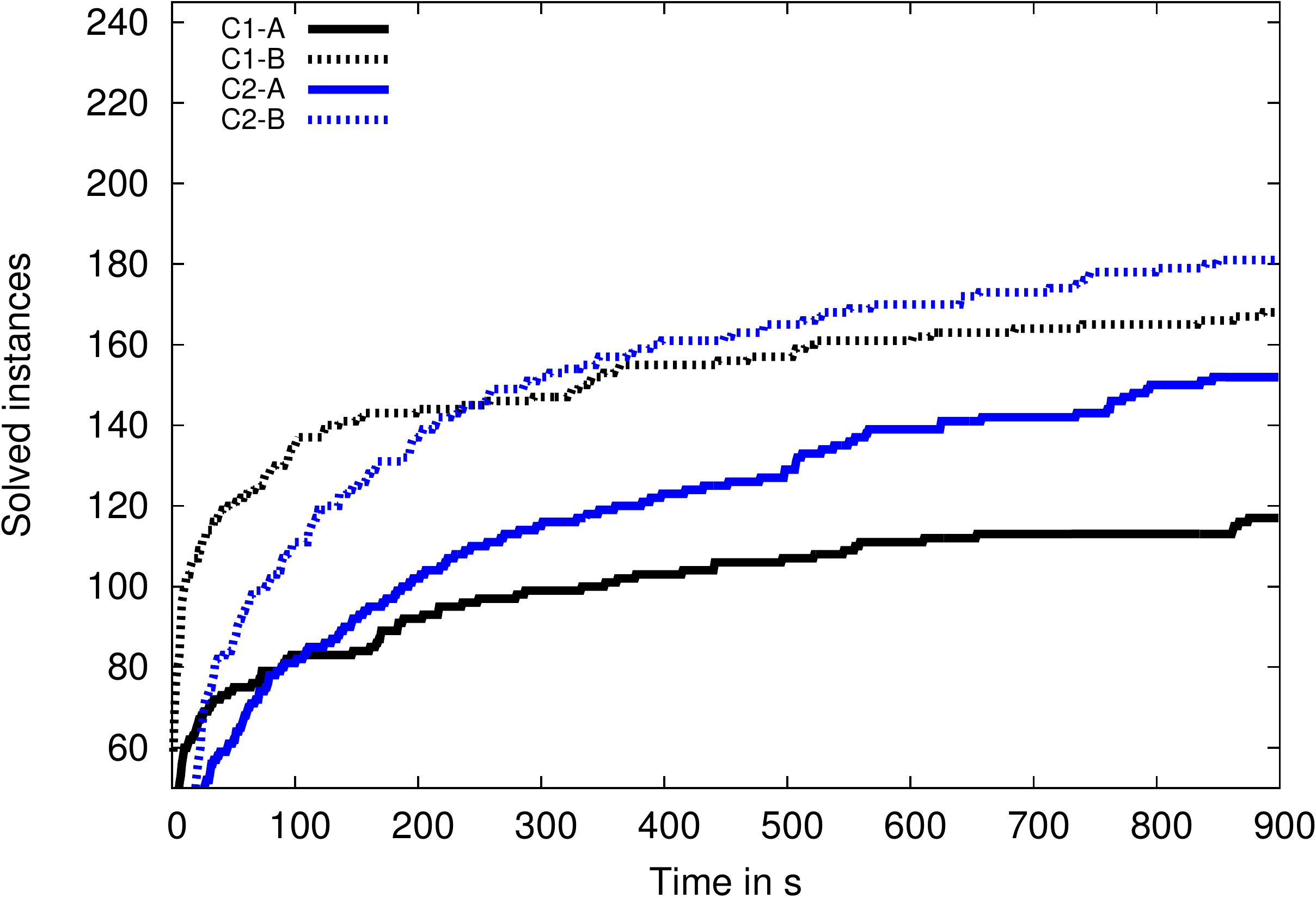}}

\subfigure[Instances $\cI^{15}$.\label{plot-15-0}]{\includegraphics[width=.5\textwidth]{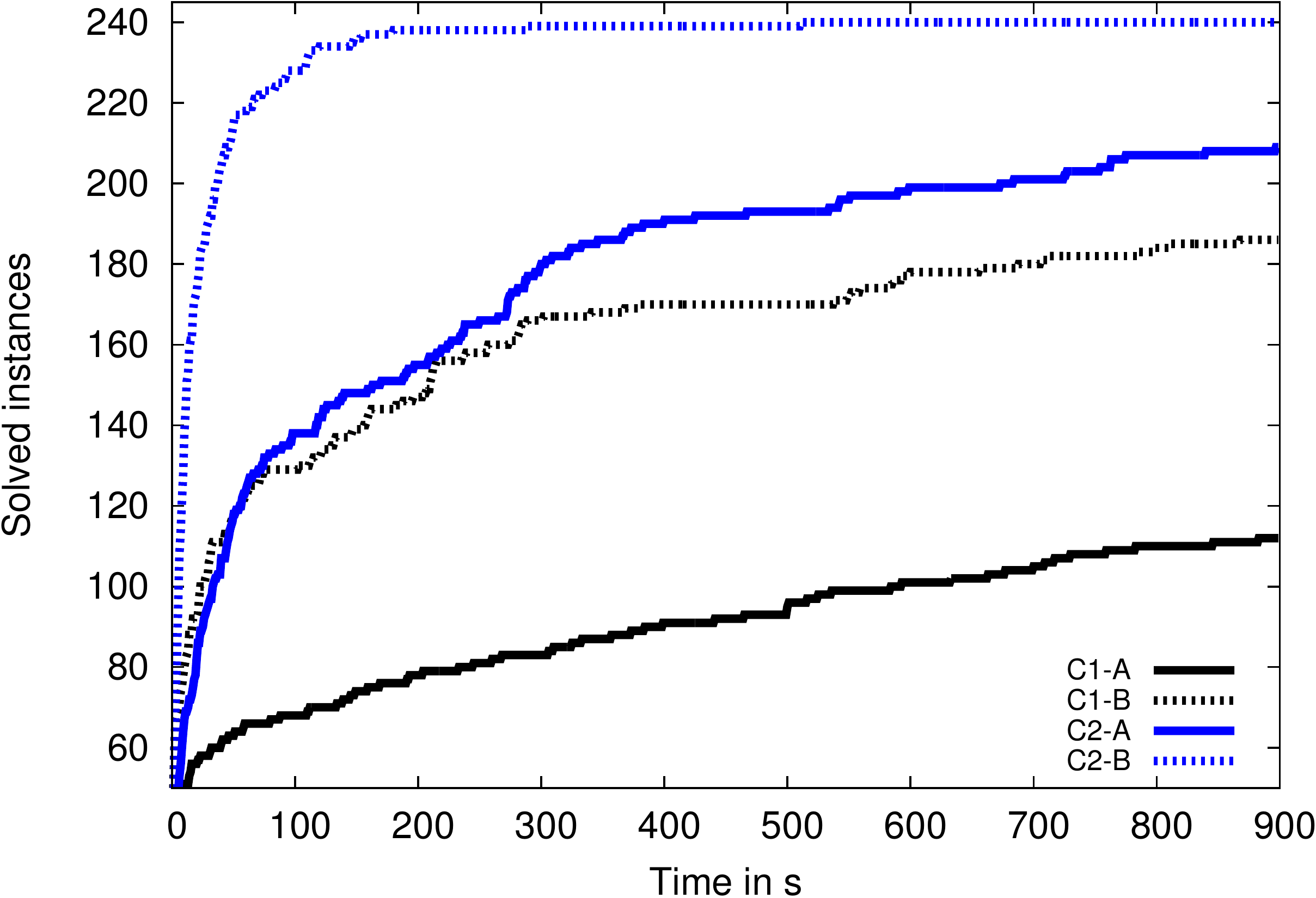}}%
\subfigure[Instances $\cI^{25}$.\label{plot-25-0}]{\includegraphics[width=.5\textwidth]{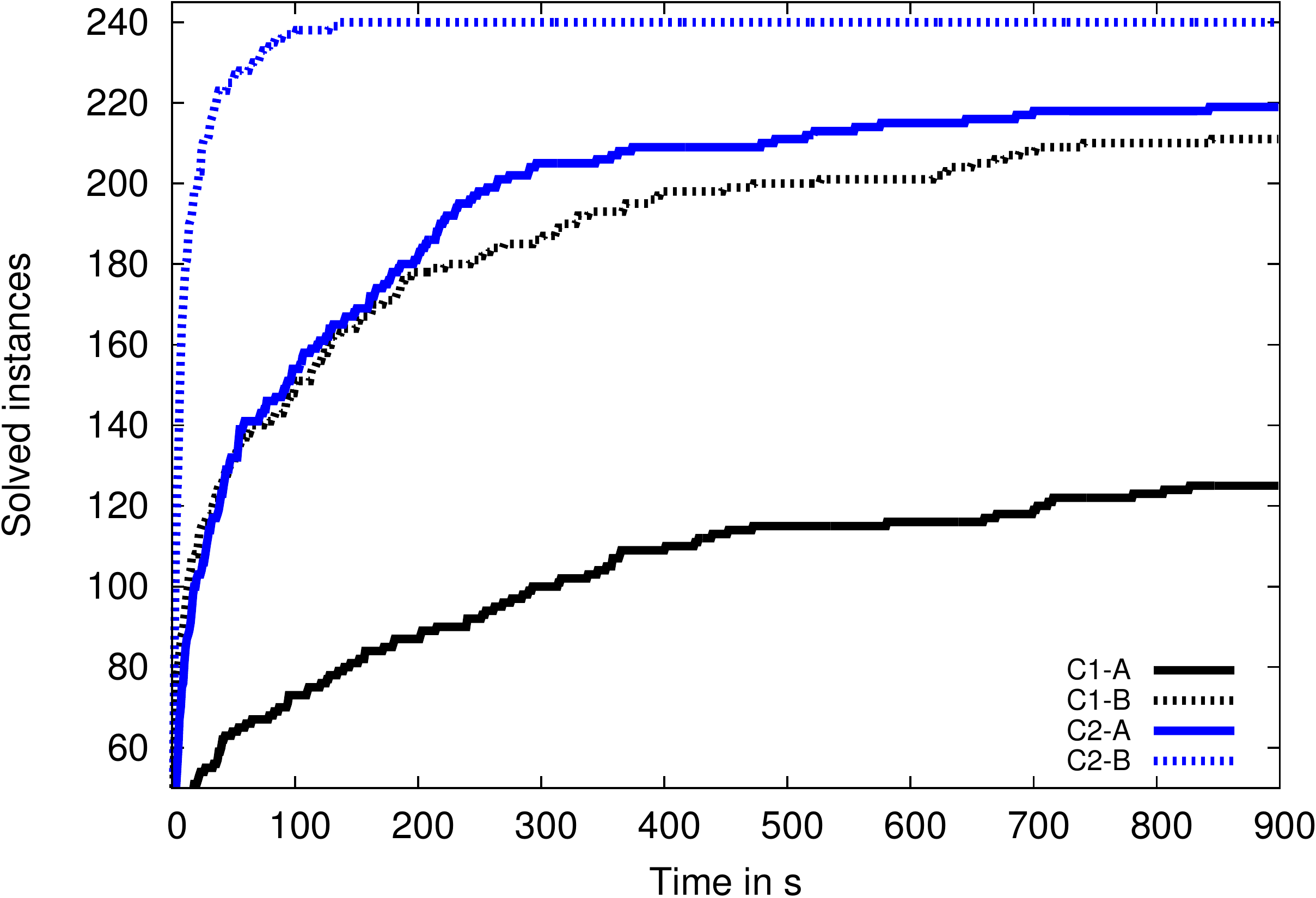}}
\caption{Performance profile for unconstrained problems.}\label{fig:ex0}
\end{figure}

The results indicate that method B clearly outperforms method A to solve subproblems. As C1 requires more cuts (and therefore the subproblem is solved more often), using the better method gives an even larger performance improvement than for C2. However, method B is numerically less stable due to the bigM constants, which are particularly large when the matrix $C$ is dense. For five instances, the subproblem could not be solved by Cplex due to numerical instability, which we counted as if the time limit of 900 seconds was reached for the purpose of this evaluation.

Furthermore, type 2 cuts outperform type 1 cuts in this experiment. Only when the density in matrix $C$ is low and for small deviation instances, there is a short advantage of C2 over C1. We present more detailed tables in Appendix~\ref{appendix}. There it can be seen that for less and smaller entries in $C$, more type 2 cuts need to be generated. In fact, if $C$ is dense enough and its values sufficiently large, only two solutions $y$ need to be constructed, namely $y=0$ and $y=1$, to solve the minmax regret problem, leading to a strong performance of C2 in these cases.

Overall, solution approach C2-B shows the best performance to solve the minmax regret problem with ellipsoidal uncertainty on the instances we considered here.

\subsection{Experiment 2: Shortest Path Problems}

We now consider the performance of our algorithms on shortest path instances. In Figure~\ref{fig:ex1all}, we present performance profile over all 720 instances, and a more differentiated view on instance classes in Figure~\ref{fig:ex1}.

\begin{figure}[htbp]
\centering
\includegraphics[width=\textwidth]{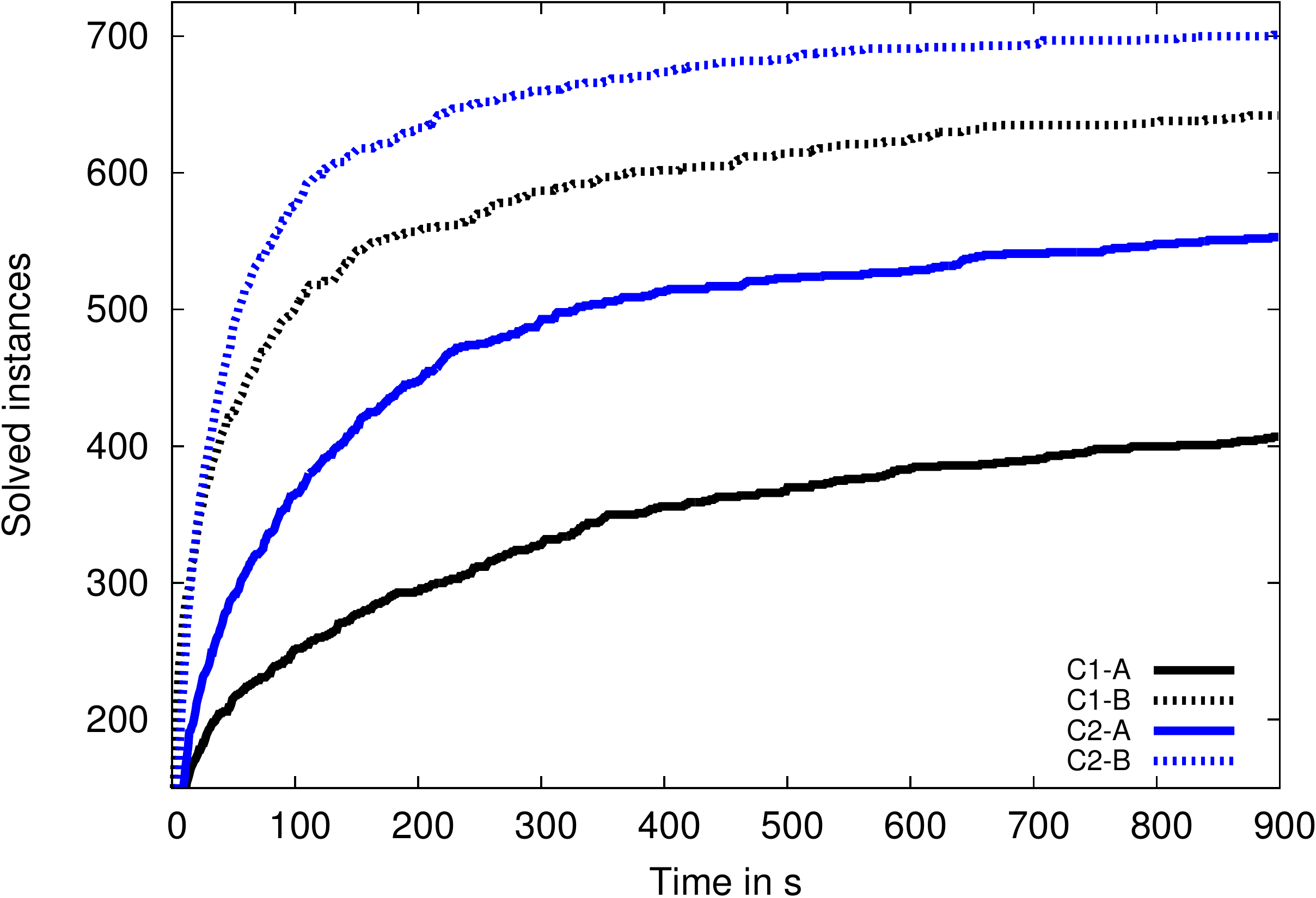}
\caption{Performance profile for shortest path problems, all instances.}\label{fig:ex1all}
\end{figure}

\begin{figure}[htbp]
\centering
\subfigure[Instances $\cJ^s$.\label{plot-small-1}]{\includegraphics[width=.5\textwidth]{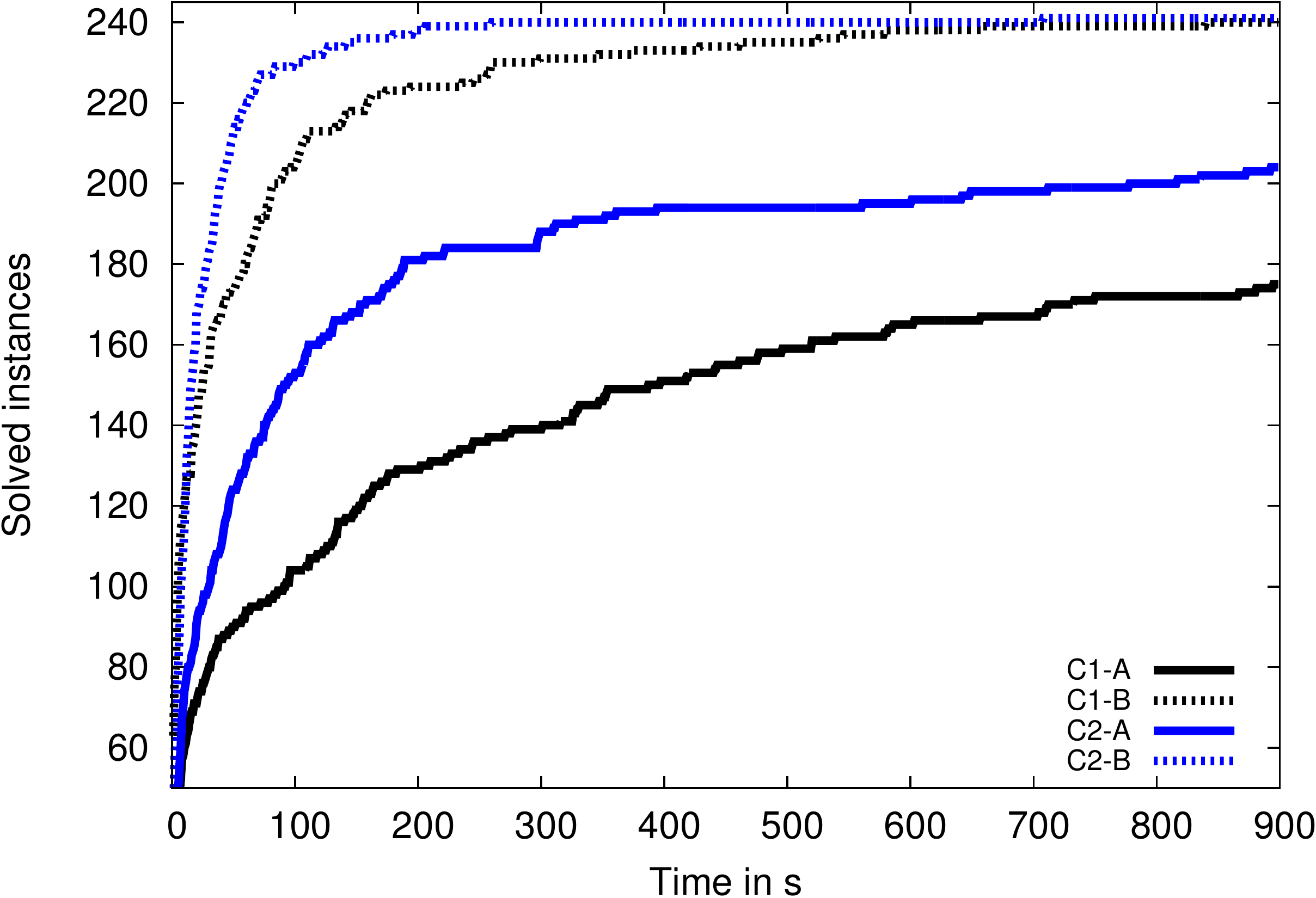}}%
\subfigure[Instances $\cJ^m$.\label{plot-medium-1}]{\includegraphics[width=.5\textwidth]{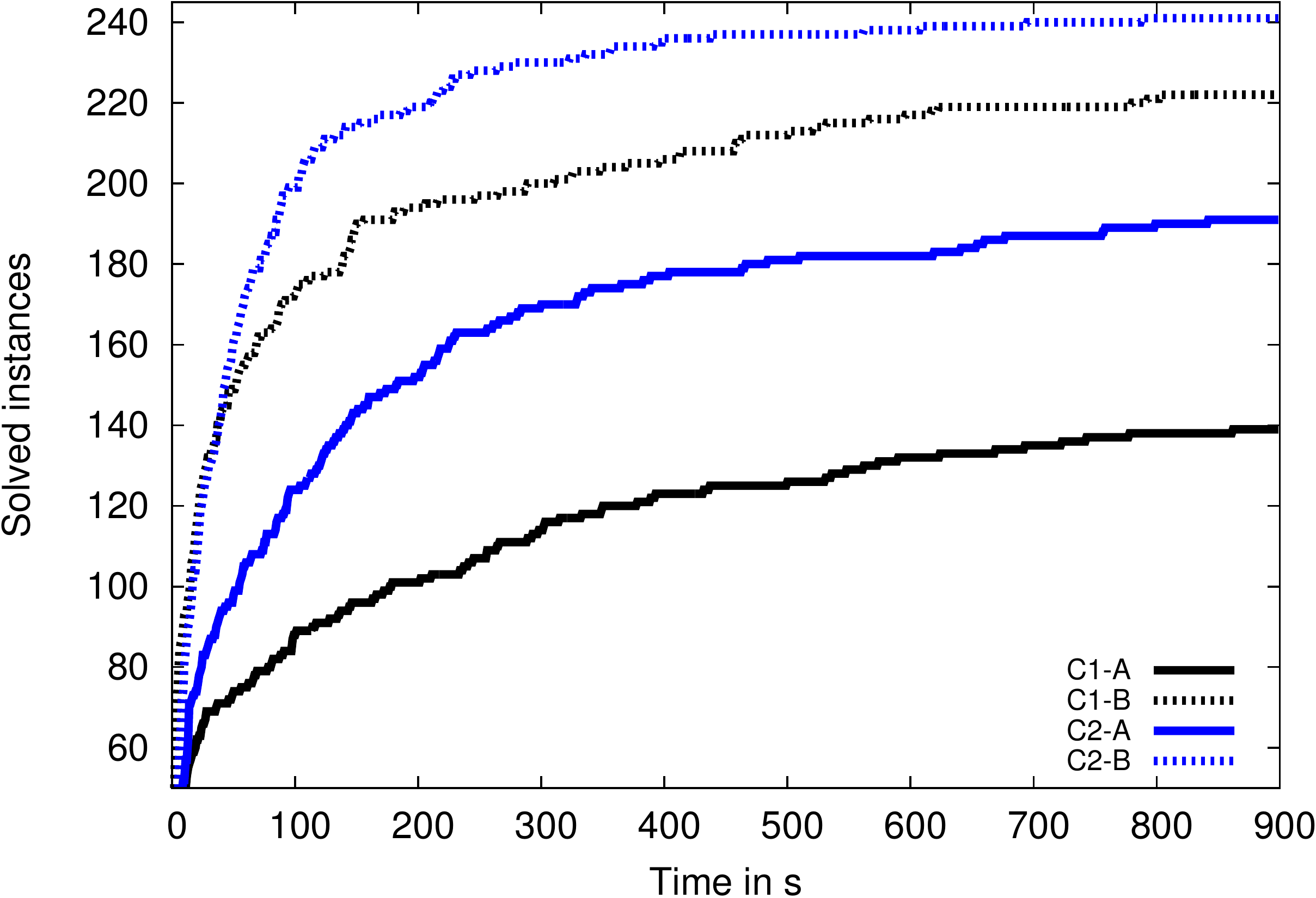}}

\subfigure[Instances $\cJ^l$.\label{plot-large-1}]{\includegraphics[width=.5\textwidth]{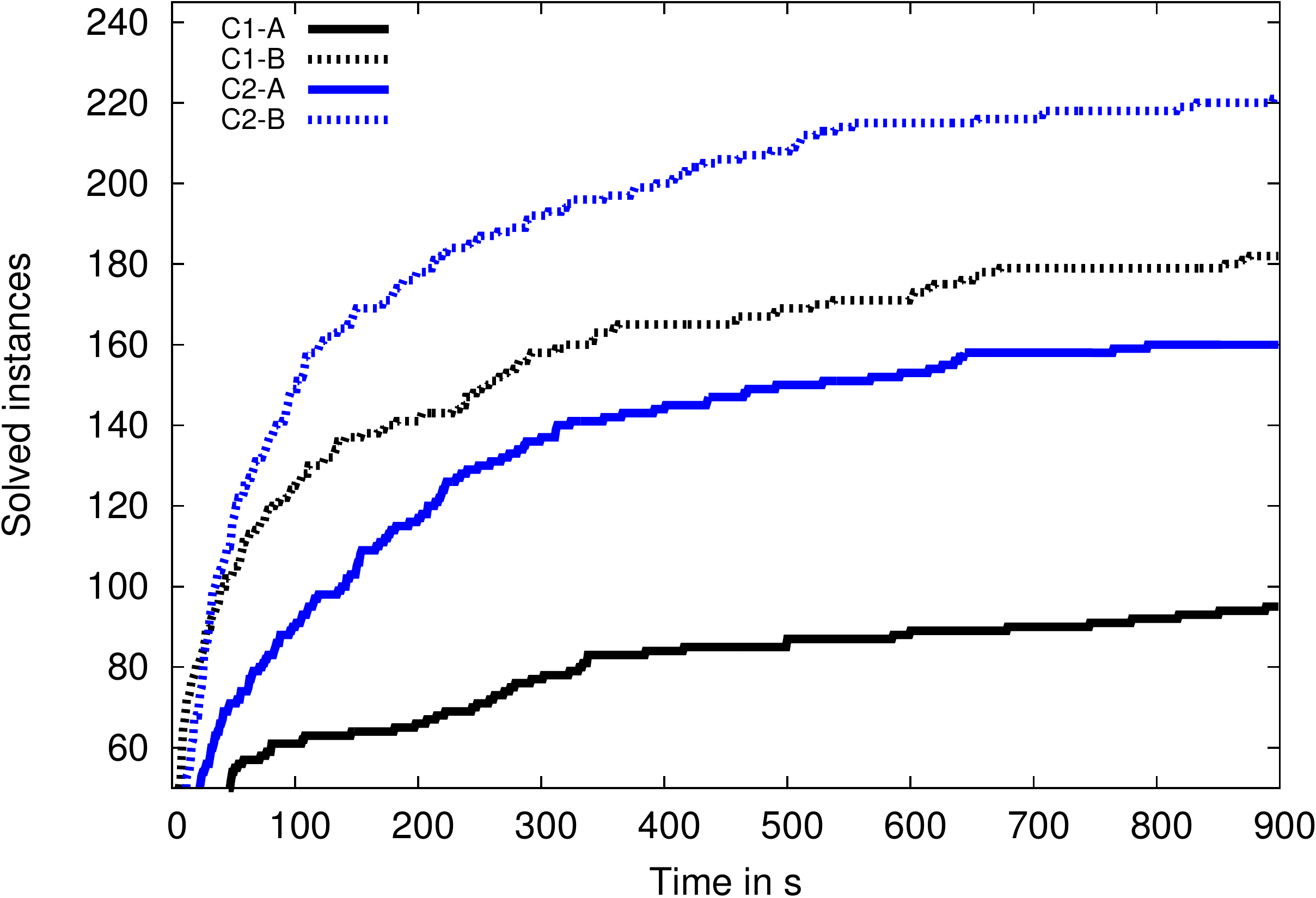}}%
\subfigure[Instances $\cJ^5$.\label{plot-5-1}]{\includegraphics[width=.5\textwidth]{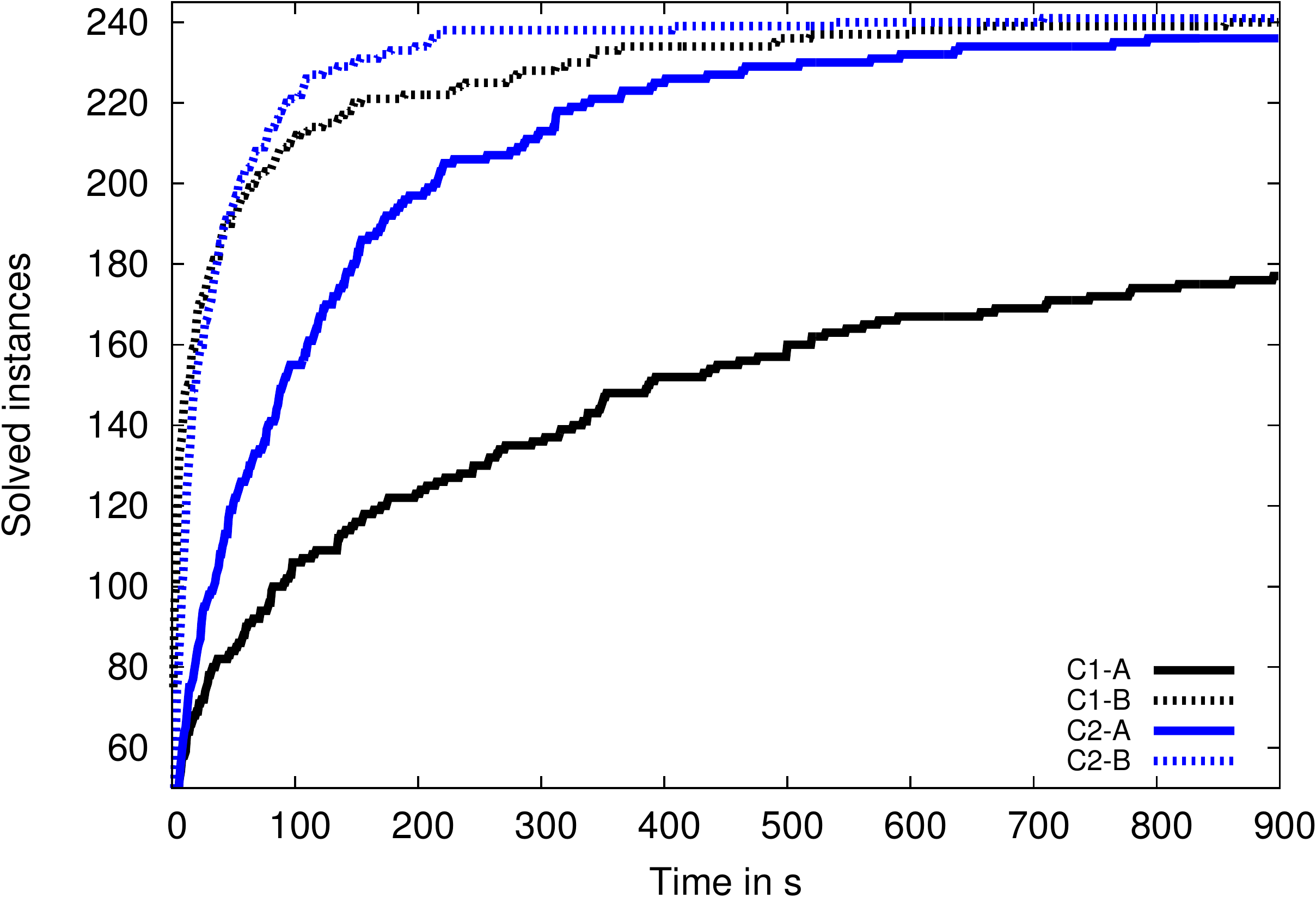}}

\subfigure[Instances $\cJ^{15}$.\label{plot-15-1}]{\includegraphics[width=.5\textwidth]{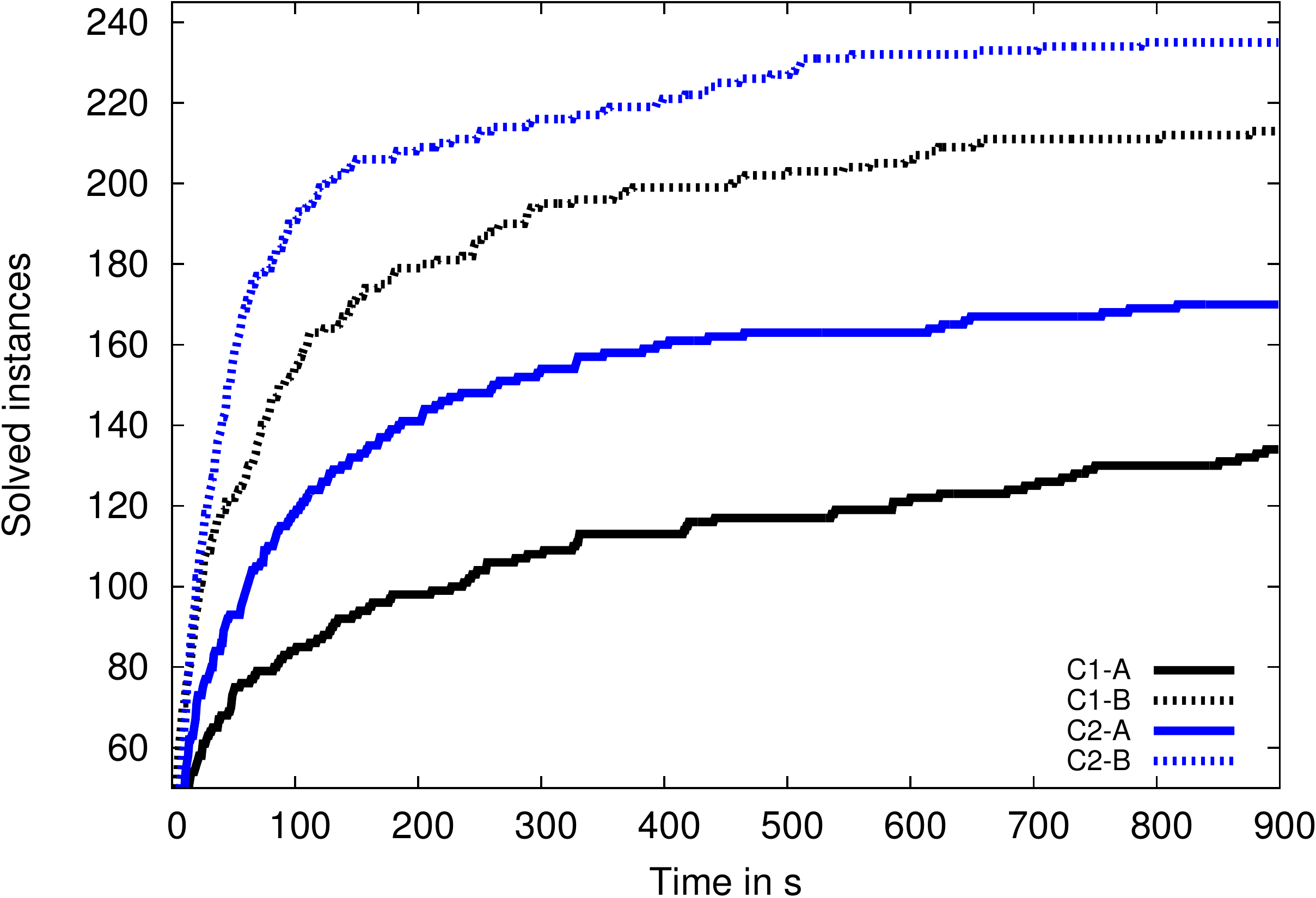}}%
\subfigure[Instances $\cJ^{25}$.\label{plot-25-1}]{\includegraphics[width=.5\textwidth]{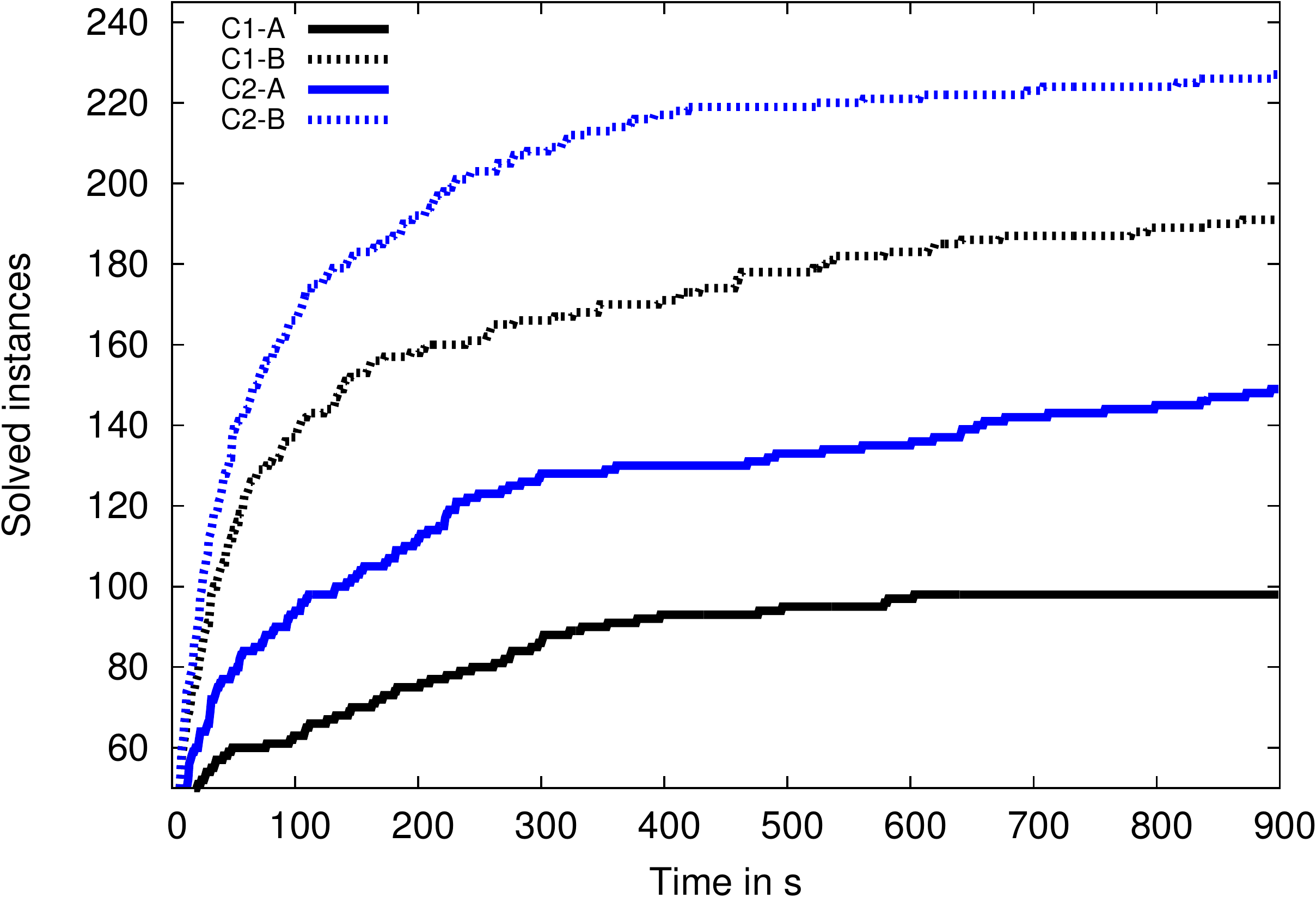}}
\caption{Performance profile for shortest path problems.}\label{fig:ex1}
\end{figure}

In this case, the strong performance of C2 for high-density matrices $C$ with large values that could be observed for unconstrained instances cannot be observed. The reason for this is that it does not suffice to generate the two cuts $y=0$ and $y=1$, as these are infeasible in this setting. Hence, performance of C2 actually deteriorates if the density of $C$ or the size of the values in $C$ increase.

The relative order of the methods, i.e., subproblems B perform better than A and cuts of type 2 perform better than cuts of type 1 is the same as before. Hence, also for these shortest path problems, we find that the best solution approach is given by C2-B. More detailed tables are given in Appendix~\ref{appendix}.

\section{Conclusion}
\label{sec:conclusion}

Minmax regret problems are a cornerstone in robust optimization. Despite their popularity, research has been focusing on only very simple uncertainty sets, which might not reflect actual requirements in real-world problems. In this work, we considered minmax regret problems with ellipsoidal uncertainty sets.

We gave a thorough discussion of arising problem complexities for the unconstrained combinatorial problem, and the shortest path problem. To solve these problems, two types of cuts that can be used in a scenario relaxation procedure were derived, as well as two linearizations to solve the subproblem of generating new cuts. We compared the performance of these methods in two computational experiments, using unconstrained and shortest path problems as a testbed. 

We found that the increased complexity of master problems with type 2 cuts are worth the effort, as less iterations are required to solve the minmax regret problem to optimality. The advantage is particularly strong for the unconstrained problem if the values of the deviation matrix $C$ are dense and large.

In future research, heuristic solution algorithms should be developed and tested, due to the high computational effort when solving these problems.

\bibliographystyle{alpha}
\bibliography{references}

\newpage
\appendix

\section{Appendix}
\label{appendix}

We present detailed results for the experiments described in Section~\ref{sec:experiments}. In Tables~\ref{tab:01} and \ref{tab:11}, we show the average number of cuts and the number of problems that were solved to optimality for each instance set. We show how much time was spent in the relaxed master problem and in the subproblem in Tables~\ref{tab:02} and~\ref{tab:12}.

\begin{table}[htbp]
\small\centering
\begin{tabular}{r|rr|rr|rr|rr}
 & \multicolumn{2}{|c|}{C1-A} & \multicolumn{2}{|c|}{C1-B} & \multicolumn{2}{|c|}{C2-A} & \multicolumn{2}{|c}{C2-B} \\
Inst. & Cuts & Opt & Cuts & Opt & Cuts & Opt & Cuts & Opt \\
\hline
 & & & & & & & & \\[-2ex]
$\cI^{5\phantom{0}}_s$ & 57.0 & 45 & 160.1 & 65 & 12.7 & 43 & 13.6 & 50 \\[1mm]
$\cI^{15}_s$ & 18.5 & 38 & 48.6 & 72 & 3.6 & 60 & 3.9 & 80 \\[1mm]
$\cI^{25}_s$ & 9.2 & 41 & 19.3 & 77 & 2.7 & 59 & 2.9 & 79 \\[1mm]
\hline
 & & & & & & & & \\[-2ex]
$\cI^{5\phantom{0}}_m$ & 33.2 & 39 & 142.9 & 53 & 9.0 & 46 & 10.2 & 53 \\[1mm]
$\cI^{15}_m$ & 20.5 & 35 & 75.5 & 57 & 3.7 & 68 & 3.7 & 79 \\[1mm]
$\cI^{25}_m$ & 21.5 & 43 & 53.7 & 75 & 2.5 & 79 & 2.5 & 80 \\[1mm]
\hline
 & & & & & & & & \\[-2ex]
$\cI^{5\phantom{0}}_l$ & 26.6 & 32 & 136.3 & 49 & 5.0 & 62 & 5.5 & 77 \\[1mm]
$\cI^{15}_l$ & 25.4 & 38 & 92.9 & 56 & 2.3 & 80 & 2.3 & 80 \\[1mm]
$\cI^{25}_l$ & 32.7 & 40 & 77.4 & 58 & 2.0 & 80 & 2.0 & 80
 \end{tabular}
\caption{Results for unconstrained instances. ''Cuts'' is the average number of cuts that were generated during the solution process. ''Opt'' is the number of problems that were solved to optimality, out of 80 for each instance type.}\label{tab:01}
\end{table}

\begin{table}[htbp]
\small\centering
\begin{tabular}{r|rr|rr|rr|rr}
 & \multicolumn{2}{|c|}{C1-A} & \multicolumn{2}{|c|}{C1-B} & \multicolumn{2}{|c|}{C2-A} & \multicolumn{2}{|c}{C2-B} \\
Inst. & Main & SUB & Main & SUB & Main & SUB & Main & SUB \\
\hline
 & & & & & & & & \\[-2ex]
$\cI^{5\phantom{0}}_s$ & 5.8 & 94.2 & 55.8 & 44.2 & 73.5 & 26.5 & 97.8 & 2.2 \\[1mm]
$\cI^{15}_s$ & 2.0 & 98.0 & 8.8 & 91.2 & 21.6 & 78.4 & 65.7 & 34.3 \\[1mm]
$\cI^{25}_s$ & 1.3 & 98.7 & 5.0 & 95.0 & 12.6 & 87.4 & 43.6 & 56.4 \\[1mm]
\hline
 & & & & & & & & \\[-2ex]
$\cI^{5\phantom{0}}_m$ & 4.0 & 96.0 & 48.5 & 51.5 & 60.6 & 39.4 & 95.4 & 4.6 \\[1mm]
$\cI^{15}_m$ & 2.7 & 97.3 & 16.8 & 83.2 & 25.2 & 74.8 & 73.0 & 27.0 \\[1mm]
$\cI^{25}_m$ & 2.5 & 97.5 & 17.3 & 82.7 & 18.6 & 81.4 & 70.4 & 29.6 \\[1mm]
\hline
 & & & & & & & & \\[-2ex]
$\cI^{5\phantom{0}}_l$ & 3.2 & 96.8 & 45.9 & 54.1 & 37.9 & 62.1 & 93.1 & 6.9 \\[1mm]
$\cI^{15}_l$ & 2.2 & 97.8 & 46.4 & 53.6 & 24.5 & 75.5 & 86.0 & 14.0 \\[1mm]
$\cI^{25}_l$ & 4.6 & 95.4 & 64.1 & 35.9 & 27.5 & 72.5 & 90.8 & 9.2
 \end{tabular}
\caption{Results for unconstrained instances. ''Main'' is the average percentage of time that was spent in the master problem. ''SUB'' is the average percentage of time that was spent in the subproblem.}\label{tab:02}
\end{table}

 \begin{table}[htbp]
\small\centering
\begin{tabular}{r|rr|rr|rr|rr}
 & \multicolumn{2}{|c|}{C1-A} & \multicolumn{2}{|c|}{C1-B} & \multicolumn{2}{|c|}{C2-A} & \multicolumn{2}{|c}{C2-B} \\
Inst. & Cuts & Opt & Cuts & Opt & Cuts & Opt & Cuts & Opt \\
\hline
 & & & & & & & & \\[-2ex]
$\cJ^{5\phantom{0}}_s$ & 11.4 & 69 & 13.2 & 80 & 2.6 & 79 & 2.6 & 80 \\[1mm]
$\cJ^{15}_s$ & 10.0 & 59 & 17.0 & 80 & 2.4 & 65 & 2.6 & 80 \\[1mm]
$\cJ^{25}_s$ & 8.9 & 46 & 19.7 & 79 & 2.1 & 59 & 2.7 & 80 \\[1mm]
\hline
 & & & & & & & & \\[-2ex]
$\cJ^{5\phantom{0}}_m$ & 15.1 & 62 & 21.0 & 80 & 2.8 & 80 & 2.8 & 80 \\[1mm]
$\cJ^{15}_m$ & 15.2 & 45 & 43.1 & 78 & 3.0 & 60 & 3.8 & 80 \\[1mm]
$\cJ^{25}_m$ & 12.4 & 31 & 63.2 & 63 & 2.6 & 50 & 4.0 & 80 \\[1mm]
\hline
 & & & & & & & & \\[-2ex]
$\cJ^{5\phantom{0}}_l$ & 24.0 & 45 & 57.2 & 79 & 3.7 & 76 & 3.8 & 80 \\[1mm]
$\cJ^{15}_l$ & 15.9 & 29 & 80.1 & 54 & 2.7 & 44 & 4.9 & 74 \\[1mm]
$\cJ^{25}_l$ & 14.7 & 20 & 79.8 & 48 & 2.8 & 39 & 5.2 & 66
 \end{tabular}
\caption{Results for shortest path instances. ''Cuts'' is the average number of cuts that were generated during the solution process. ''Opt'' is the number of problems that were solved to optimality, out of 80 for each instance type..}\label{tab:11}
\end{table}

 \begin{table}[htbp]
\small\centering
\begin{tabular}{r|rr|rr|rr|rr}
 & \multicolumn{2}{|c|}{C1-A} & \multicolumn{2}{|c|}{C1-B} & \multicolumn{2}{|c|}{C2-A} & \multicolumn{2}{|c}{C2-B} \\
Inst. & Main & SUB & Main & SUB & Main & SUB & Main & SUB \\
\hline
 & & & & & & & & \\[-2ex]
$\cJ^{5\phantom{0}}_s$ & 0.8 & 99.2 & 17.4 & 82.6 & 23.6 & 76.4 & 83.1 & 16.9 \\[1mm]
$\cJ^{15}_s$ & 0.8 & 99.2 & 2.5 & 97.5 & 22.4 & 77.6 & 57.7 & 42.3 \\[1mm]
$\cJ^{25}_s$ & 0.7 & 99.3 & 1.9 & 98.1 & 17.3 & 82.7 & 50.6 & 49.4 \\[1mm]
\hline
 & & & & & & & & \\[-2ex]
$\cJ^{5\phantom{0}}_m$ & 0.7 & 99.3 & 11.5 & 88.5 & 32.1 & 67.9 & 85.3 & 14.7 \\[1mm]
$\cJ^{15}_m$ & 1.1 & 98.9 & 4.9 & 95.1 & 34.2 & 65.8 & 80.9 & 19.1 \\[1mm]
$\cJ^{25}_m$ & 0.9 & 99.1 & 5.4 & 94.6 & 27.7 & 72.3 & 82.2 & 17.8 \\[1mm]
\hline
 & & & & & & & & \\[-2ex]
$\cJ^{5\phantom{0}}_l$ & 1.4 & 98.6 & 9.2 & 90.8 & 44.0 & 56.0 & 86.6 & 13.4 \\[1mm]
$\cJ^{15}_l$ & 1.3 & 98.7 & 9.3 & 90.7 & 27.6 & 72.4 & 87.4 & 12.6 \\[1mm]
$\cI^{25}_l$ & 0.8 & 99.2 & 9.6 & 90.4 & 19.0 & 81.0 & 78.0 & 22.0
 \end{tabular}
\caption{Results for shortest path instances. ''Main'' is the average percentage of time that was spent in the master problem. ''SUB'' is the average percentage of time that was spent in the subproblem.}\label{tab:12}
\end{table}

\end{document}